\documentclass[a4paper,notitlepage, eqno]{article}%
\usepackage{amssymb}
\usepackage{graphicx}
\usepackage{amsmath}
\usepackage{amsthm}
\usepackage{amsfonts}%
\providecommand{\U}[1]{\protect\rule{.1in}{.1in}}
\theoremstyle{plain}
\newtheorem{theorem}{Theorem}[section]

\newtheorem{cor}[theorem]{Corollary}

\newtheorem{prop}[theorem]{Proposition}

\theoremstyle{definition}

\newtheorem{rem}[theorem]{Remark}
\numberwithin{equation}{section}
\numberwithin{theorem}{section}

\ifx\pdfoutput\relax\let\pdfoutput=\undefined\fi
\newcount\msipdfoutput
\ifx\pdfoutput\undefined\else
\ifcase\pdfoutput\else
\ifx\paperwidth\undefined\else
\ifdim\paperheight=0pt\relax\else\pdfpageheight\paperheight\fi
\ifdim\paperwidth=0pt\relax\else\pdfpagewidth\paperwidth\fi
\fi\fi\fi
\begin{document}

\title{Nonnegative solutions of an indefinite sublinear Robin problem II: local and
global exactness results  }
\author{Uriel Kaufmann\thanks{FaMAF-CIEM (CONICET), Universidad Nacional de
C\'{o}rdoba, Medina Allende s/n, Ciudad Universitaria, 5000 C\'{o}rdoba,
Argentina. \textit{E-mail address: }kaufmann@mate.uncor.edu} , Humberto Ramos
Quoirin \thanks{CIEM-FaMAF, Universidad Nacional de C\'{o}rdoba, (5000)
C\'{o}rdoba, Argentina. \textit{E-mail address: }humbertorq@gmail.com} ,
Kenichiro Umezu\thanks{Department of Mathematics, Faculty of Education,
Ibaraki University, Mito 310-8512, Japan. \textit{E-mail address:
}kenichiro.umezu.math@vc.ibaraki.ac.jp}
\and \noindent}
\maketitle

\begin{abstract}
We go further in the investigation of the Robin problem
\[%
\begin{cases}
-\Delta u=a(x)u^{q} & \mbox{in $\Omega$},\\
u\geq0 & \mbox{in $\Omega$},\\
\partial_{\nu}u=\alpha u & \mbox{on $\partial \Omega$},
\end{cases}
\leqno{(P_{\alpha})}
\]
on a bounded domain $\Omega\subset\mathbb{R}^{N}$, with $a$ sign-changing and
$0<q<1$. Assuming the existence of a positive solution for $\alpha=0$ (which
holds if $q$ is close enough to 1), we sharpen the description of the
nontrivial solution set of $(P_{\alpha})$ for $\alpha>0$. Moreover,
strengthening the assumptions on $a$ and $q$ we provide a global (i.e. for
\textit{every} $\alpha>0$) exactness result on the number of solutions of
$(P_{\alpha})$ . Our approach also applies to the problem
\[%
\begin{cases}
-\Delta u=\alpha u+a(x)u^{q} & \mbox{in $\Omega$},\\
u\geq0 & \mbox{in $\Omega$},\\
\partial_{\nu}u=0 & \mbox{on $\partial \Omega$}.
\end{cases}
\leqno{(S_\alpha)}
\]

\end{abstract}

\section{Introduction}

\label{sec:Intro}

In this article we proceed with the study of the indefinite problem
\[%
\begin{cases}
-\Delta u=a(x)u^{q} & \mbox{in $\Omega$},\\
u\geq0 & \mbox{in $\Omega$},\\
\partial_{\nu}u=\alpha u & \mbox{on $\partial \Omega$}.
\end{cases}
\leqno{(P_\alpha)}
\]
Here $\Omega\subset\mathbb{R}^{N}$ ($N\geq1$) is a smooth bounded domain,
$\partial_{\nu}:=\frac{\partial}{\partial\nu}$, $\nu$ is the unit outward
normal to $\partial\Omega$, $a\in C^{\theta}(\overline{\Omega})$ ($0<\theta
<1$) changes sign, $0<q<1$ and $\alpha\geq0$.

In \cite{KRQU2019} we have studied the structure of the nontrivial nonnegative
solutions set of this problem, considered with respect to $\alpha\in
\lbrack-\infty,\infty)$. Let us recall that the case $\alpha=-\infty$ is
understood as the Dirichlet boundary condition $u=0$ on $\partial\Omega$. We
shall go deeper in this analysis, giving a more precise description of the
solutions set of $(P_{\alpha})$ in the region $\alpha>0$ and weakening some of
the assumptions in \cite{KRQU2019}.

By a \textit{solution} of $(P_{\alpha})$ we mean a classical nonnegative
solution $u\in C^{2+r}(\overline{\Omega})$, for some $r\in(0,1)$. We say that
$u$ is \textit{nontrivial} if $u\not \equiv 0$. In particular, we are
interested in solutions lying in
\[
\mathcal{P}^{\circ}:=\{u\in C(\overline{\Omega}):u>0\ \mbox{on}\ \overline
{\Omega}\}.
\]
Note that
since $q\in\left(  0,1\right)  $ and $a$ changes sign in $\Omega$, this
problem may have solutions that do not belong to $\mathcal{P}^{\circ}$, see
e.g. \cite[Remark 3.7]{KRQU2019}.

The following condition on $a$, which is known to be necessary for the
existence of solutions in $\mathcal{P}^{\circ}$ when $\alpha\geq0$, shall be
assumed throughout this paper:
\[
\int_{\Omega}a<0.\leqno{({\bf A.0})}
\]
In addition, we shall assume a technical condition on the set
\[
\Omega_{+}^{a}:=\{x\in\Omega:a(x)>0\},
\]
namely:
\[
\Omega_{+}^{a}\text{ has \textit{finitely }many connected components, which
are all smooth.}\leqno{({\bf A.1})}
\]





Our results in \cite{KRQU2019} for $\alpha>0$ have been established when
$q\in\mathcal{A}_{\mathcal{N}}:=\mathcal{A}_{0}$, where
\[
\mathcal{A}_{\alpha}=\mathcal{A}_{\alpha}(a):=\{q\in
(0,1):\mbox{any nontrivial solution of $(P_\alpha)$ lies in $\mathcal{P}^\circ$}\}.
\]
We recall that, under $(A.0)$ and $(A.1)$, we have $\mathcal{A}_{\mathcal{N}%
}=(q_{\mathcal{N}},1)$ for some $q_{\mathcal{N}}=q_{\mathcal{N}}(a) \in
\lbrack0,1)$. Moreover, if $q\in\mathcal{A}_{\mathcal{N}}$ then $(P_{0})$ has
a \textit{unique} nontrivial solution $u_{\mathcal{N}}$, which satisfies
$u_{\mathcal{N}}\in\mathcal{P}^{\circ}$ \cite[Theorem 1.9]{KRQU16}, and every
nontrivial solution of $(P_{\alpha})$ belongs to $\mathcal{P}^{\circ}$ if
$0<\alpha<\alpha_{+}$, for some $\alpha_{+}>0$ \cite[Proposition
2.3]{KRQU2019}.

In the present paper, most of our results shall be proved under the weaker
assumption $q\in\mathcal{I}_{\mathcal{N}}$, where
\[
\mathcal{I}_{\mathcal{N}}=\mathcal{I}_{\mathcal{N}}(a):=\{q\in
(0,1):\mbox{$(P_0)$ has a solution in
$\mathcal{P}^\circ$}\}.
\]
Under $(A.0)$, it is known that $(P_{0})$ has at least one nontrivial solution
for any $q\in(0,1)$, see e.g. \cite[Theorem 2.1]{BPT2}, so $\mathcal{A}%
_{\mathcal{N}}\subseteq\mathcal{I}_{\mathcal{N}}$. In general, we may have
$\mathcal{A}_{\mathcal{N}}\not =\mathcal{I}_{\mathcal{N}}$; however when
$\Omega_{a}^{+}$ is connected and smooth, it holds that $\mathcal{A}%
_{\mathcal{N}}=\mathcal{I}_{\mathcal{N}}$, see \cite[Theorem 1.4]{KRQUnodea}.
Also, by \cite[Theorem 3.1]{BPT2} we know that $(P_{0})$ has at most one
solution which is positive in $\Omega_{+}^{a}$. Thus, whenever $q\in
\mathcal{I}_{\mathcal{N}}$ we denote by $u_{\mathcal{N}}(a)$ (or simply
$u_{\mathcal{N}}$ if no confusion arises) the unique solution in
$\mathcal{P}^{\circ}$ of $(P_{0})$ (or $(P_{0,a})$ if we need to stress the
dependence on $a$) .

To the best of our knowledge, prior to \cite{KRQU2019} the only work dealing
with $(P_{\alpha})$ for $\alpha>0$ is \cite{CT14}, where Chabrowski and
Tintarev established a \textit{local} multiplicity result for the related
problem
\[%
\begin{cases}
-\Delta w=\alpha a(x)w^{q} & \mbox{in}\ \Omega,\\
w\geq0 & \mbox{in}\ \Omega,\\
\partial_{\nu}w=\alpha w & \mbox{on}\ \partial\Omega,
\end{cases}
\leqno{(R_\alpha)}
\]
with $\alpha>0$ small (note that $(R_{\alpha})$ is equivalent to $(P_{\alpha
})$, after the change of variables $w=\alpha^{\frac{1}{1-q}}u$). More
precisely, by variational methods it was proved in \cite[Propositions 7.4 and
7.7]{CT14} that under $(A.0)$, $(R_{\alpha})$ has at least two nontrivial
solutions $w_{1,\alpha},w_{2,\alpha}$ such that $w_{1,\alpha}<w_{2,\alpha}$ on
$\overline{\Omega}$ for $\alpha>0$ small enough. Moreover, if
\begin{equation}
c_{a}:=\left(  \frac{-\int_{\Omega}a}{|\partial\Omega|}\right)  ^{\frac
{1}{1-q}}, \label{def:ca}%
\end{equation}
the following asymptotic profiles of $w_{1,\alpha},w_{2,\alpha}$ hold as
$\alpha\rightarrow0^{+}$:
\begin{equation}
w_{2,\alpha}\rightarrow c_{a}\quad\mbox{and}\quad w_{1,\alpha}\rightarrow
0\ \mbox{in}\ H^{1}(\Omega)\quad\mbox{as}\quad\alpha\rightarrow0^{+},
\label{w1w2}%
\end{equation}
and every sequence $\alpha_{n}\rightarrow0$ has a subsequence (still denoted
by the same notation) satisfying
\begin{equation}
\alpha_{n}^{-\frac{1}{1-q}}w_{1,\alpha_{n}}\rightarrow u_{0}%
\mbox{ in $H^{1}(\Omega)$}, \label{w1ap}%
\end{equation}
where $u_{0}$ is a nontrivial solution of $(P_{0})$.

In particular, we see that under $(A.0)$ the problem $(P_{\alpha})$ has, for
any $q\in(0,1)$ and $\alpha>0$ small enough, a solution in $\mathcal{P}%
^{\circ}$, namely, $\alpha^{-\frac{1}{1-q}}w_{2,\alpha}$. Our first result
asserts that $(A.0)$ is also \textit{necessary} for the existence of solutions
of $(P_{\alpha})$ that are positive in $\Omega_{+}^{a}$ (for any $\alpha\geq0$
and $0<q<1$) and gives a sufficient condition for such solutions to be in
$\mathcal{P}^{\circ}$:

\begin{theorem}
\label{tp1}If $(P_{\alpha})$ has a solution $u$ such that $u>0$ in $\Omega
_{a}^{+}$ for some $\alpha\geq0$ and $q\in(0,1)$, then $(A.0)$ holds. If we
assume in addition $(A.1)$ and $q\in\mathcal{I}_{\mathcal{N}}$, then every
such solution belongs to $\mathcal{P}^{\circ}$.
\end{theorem}

Assuming additionally $(A.1)$ and $q\in\mathcal{A}_{\mathcal{N}}$, we proved
in
\cite[Corollary 3.13]{KRQU2019} that $w_{1,\alpha},w_{2,\alpha}$ are the only
nontrivial solutions of $(R_{\alpha})$ for $\alpha>0$ small, and in
\cite[Theorem 4.4]{KRQU2019} that $(R_{\alpha})$ has a maximal closed,
connected subset of nontrivial solutions containing $\{(\alpha,w_{1,\alpha
}),(\alpha,w_{2,\alpha})\}$ under some further assumptions on $a$. We remark
that the solutions set $\{(\alpha,u)\}$ of $(P_{\alpha})$ corresponds to
$\{(\alpha,w)\}$ of $(R_{\alpha})$ by
the change of variables $u=\alpha^{-\frac{1}{1-q}}w$ for $\alpha>0$.

Let us set
\begin{equation}
\alpha_{s}=\alpha_{s}(a,q):=\sup\{\alpha\geq
0:\mbox{$(P_\alpha)$ has a solution in $\mathcal{P}^\circ$}\}.
\label{def:alphs}%
\end{equation}
Proceeding as in \cite[Proposition 3.4]{KRQU2019} we can prove that
$\alpha_{s}<\infty$ if $q\in\mathcal{I}_{\mathcal{N}}$. We shall analyze the
solution set of $(P_{\alpha})$ in neighboorhoods of $0$ and $\alpha_{s}$. This
investigation provides us with a global description of the solutions set
structure of $(P_{\alpha})$ for $\alpha\geq0$. A more precise version of
Theorem \ref{mthm} will be provided in Section \ref{sec:pr1-2}, see Theorem
\ref{mthm2}.



\begin{theorem}
\label{mthm}

Assume $(A.0),$ $(A.1)$, and $q\in\mathcal{I}_{\mathcal{N}}$. Then:

\begin{enumerate}
\item $(P_{\alpha})$
has a solution curve $\mathcal{C}_{1}=\left\{  (\alpha,u_{1,\alpha}%
);0\leq\alpha\leq\alpha_{s}\right\}  $ such that $\alpha\mapsto u_{1,\alpha
}\in\mathcal{P}^{\circ}$ is continuous and increasing on $[0,\alpha_{s}]$ and
$C^{\infty}$ in $[0,\alpha_{s})$, with $u_{1,0}=u_{\mathcal{N}}$, and
$u_{1,\alpha_{s}}$ is the unique solution of $(P_{\alpha_{s}})$ in
$\mathcal{P}^{\circ}$. Moreover, $\mathcal{C}_{1}$ is extended to a
$C^{\infty}$ curve, say $\mathcal{C}_{1}^{\prime}$ ($\supset\mathcal{C}_{1}$),
bending \textrm{to the left} in a neighborhood of $(\alpha_{s},u_{1,\alpha
_{s}})$.

\item $(P_{\alpha})$ has a solution curve $\mathcal{C}_{2}=\left\{
(\alpha,u_{2,\alpha});0<\alpha\leq\overline{\alpha}\right\}  $ for some
$\overline{\alpha}\in(0,\alpha_{s}]$, such that $\alpha\mapsto u_{2,\alpha}%
\in\mathcal{P}^{\circ}$ is continuous and decreasing on
$(0,\overline{\alpha}]$ and $C^{\infty}$ in $(0,\overline{\alpha})$, with
$\min_{\overline{\Omega}}u_{2,\alpha}\rightarrow\infty$ as $\alpha
\rightarrow0^{+}$. Moreover:

\begin{enumerate}
\item for any interior point $(\alpha,u)\in\mathcal{C}_{1}^{\prime}%
\cup\mathcal{C}_{2}$, the solutions set of $(P_{\alpha})$ in a neighborhood of
$(\alpha,u)$ is given exactly by $\mathcal{C}_{1}^{\prime}\cup\mathcal{C}_{2}$ ;

\item for every $\alpha\in(0,\overline{\alpha})$ the solutions $u_{1,\alpha
},u_{2,\alpha}$ are strictly ordered
by $u_{2,\alpha}-u_{1,\alpha}\in\mathcal{P}^{\circ}$, and these ones are the
only solutions of $(P_{\alpha})$ in $\mathcal{P}^{\circ}$ for $\alpha>0$ small.
\end{enumerate}

\item Assume additionally that $0\not \equiv a\geq0$ in some smooth domain
$D\subset\Omega$ such that $\left\vert \partial D\cap\partial\Omega\right\vert
>0$. Then $\mathcal{C}_{1}^{\prime}$ is connected to $\mathcal{C}_{2}$ by a
component (i.e., a maximal closed, connected subset) $\mathcal{C}_{\ast}$ of
solutions of $(P_{\alpha})$ in
$[0,\alpha_{s}]\times\mathcal{P}^{\circ}$, see Figure \ref{fig19_0422compo}(i).
\end{enumerate}
\end{theorem}


\begin{figure}[tbh]
\centerline{
\includegraphics[scale=0.18]{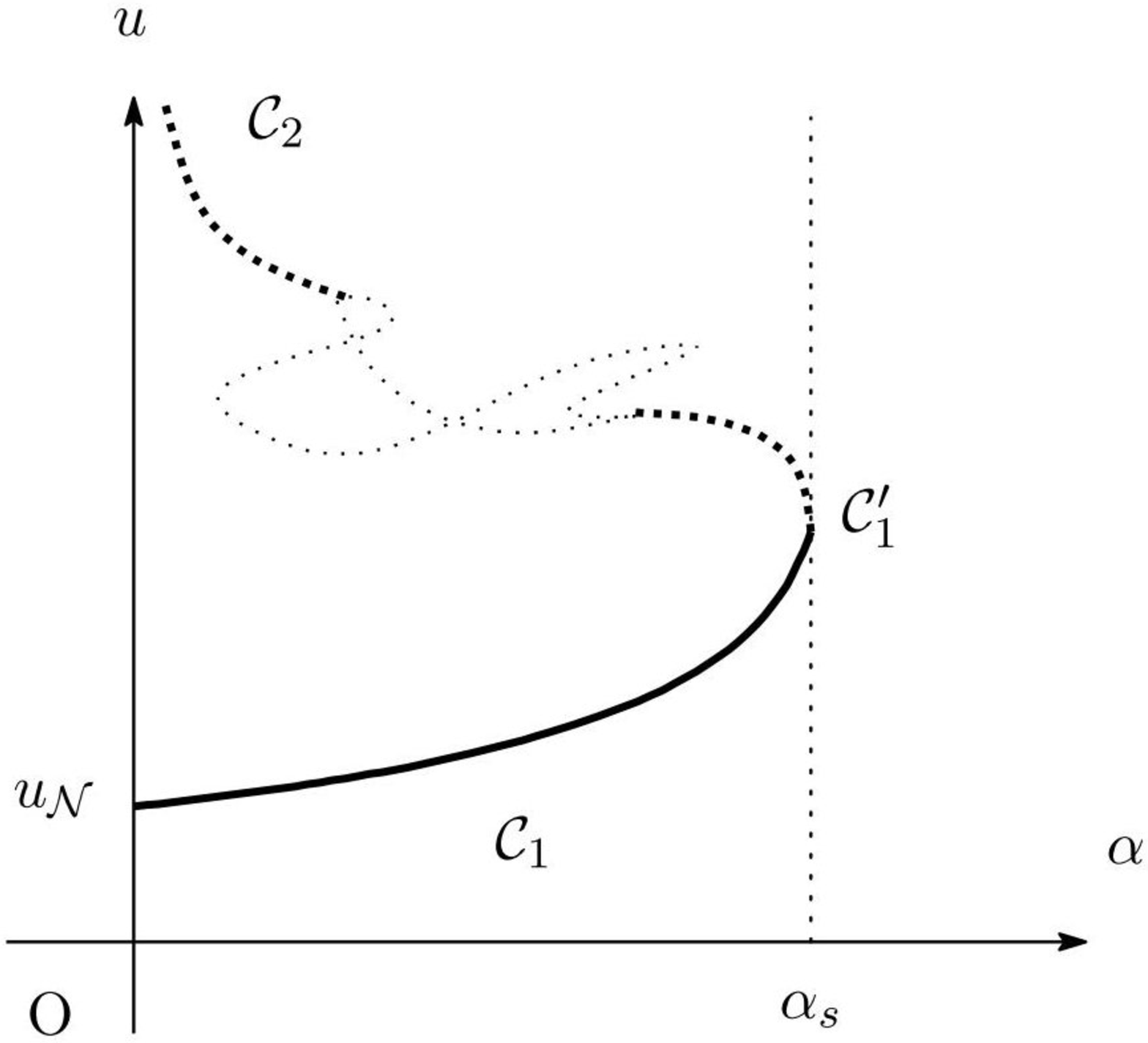}
\hskip0.35cm
\includegraphics[scale=0.18]{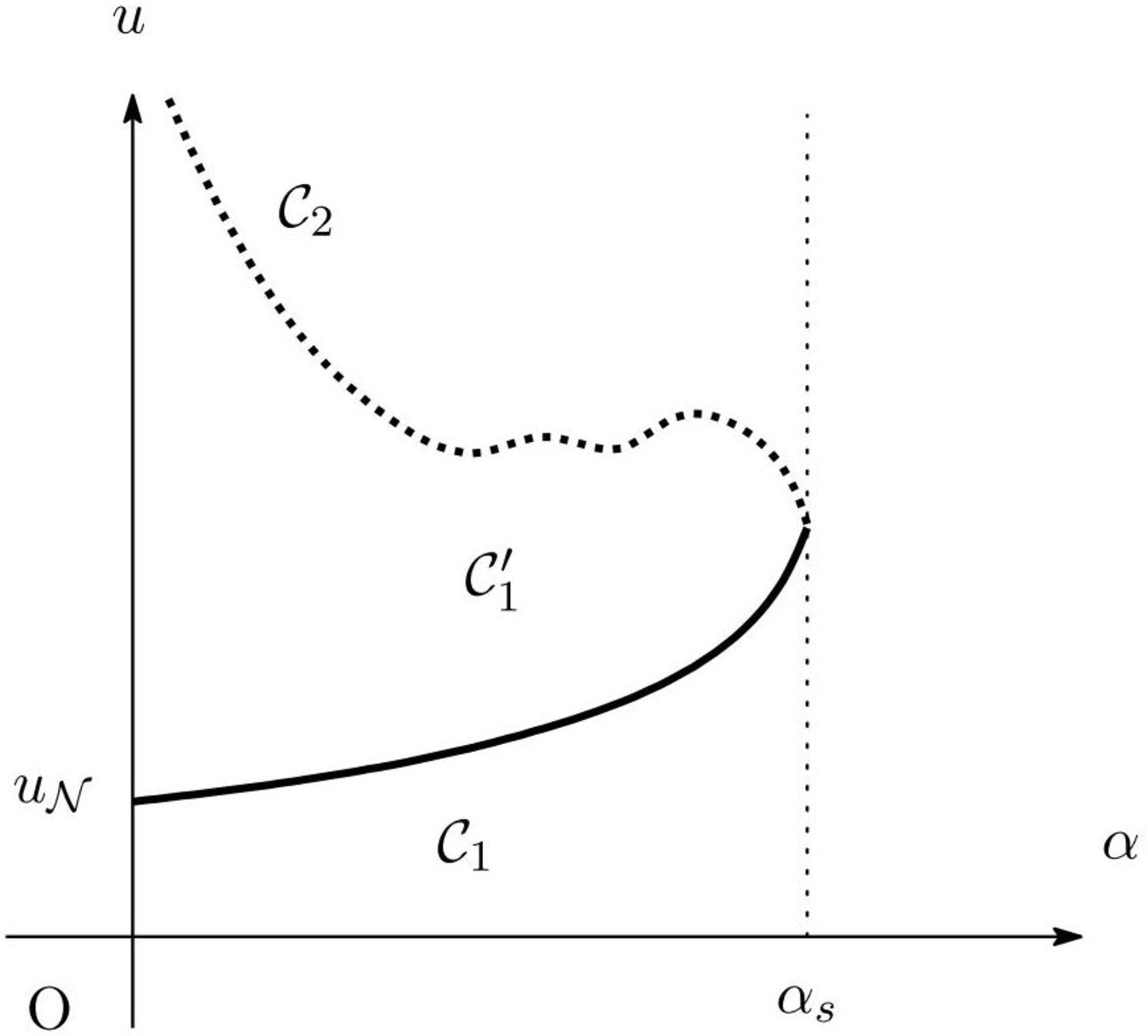}
} \centerline{(i) \hskip6.0cm (ii) }\caption{ (i) The component $\mathcal{C}%
_{\ast}$. (ii) The exact solution curve $\mathcal{C}_{1}^{\prime}$. The full
curve represents asymptotically stable solutions in $\mathcal{P}^{\circ}$,
whereas the dotted curve represents unstable solutions in $\mathcal{P}^{\circ
}$. }%
\label{fig19_0422compo}%
\end{figure}


Under different conditions to those in Theorem \ref{mthm}(iii) we shall
establish a \textit{global exactness} result. Let us denote by $\alpha
_{2}=\alpha_{2}(\Omega)$ the second (and first nontrivial) eigenvalue of the
Steklov problem%
\begin{equation}%
\begin{cases}
\Delta\phi=0 & \mbox{ in }\Omega,\\
\partial_{\nu}\phi=\alpha\phi & \mbox{ on }\partial\Omega.
\end{cases}
\label{st}%
\end{equation}
When $\Omega$ is a ball of radius $R$, we know that $\alpha_{2}(\Omega
)=\frac{1}{R}$. Also, if $\Omega$ is star shaped, several lower bounds for
$\alpha_{2}$ are known, see e.g. \cite[Section 1]{Verma} and references therein.

\begin{theorem}
\label{ball}

Assume $(A.0),$ $(A.1)$, $q\in\mathcal{I}_{\mathcal{N}}$ and
\begin{equation}
\frac{-\int_{\Omega}a}{\int_{\partial\Omega}u_{\mathcal{N}}^{1-q}}\leq
\alpha_{2}. \label{hip}%
\end{equation}
Then the following two assertions hold:

\begin{enumerate}
\item The solutions set in $\mathcal{P}^{\circ}$ of $(P_{\alpha})$ with
$\alpha\geq0$ consists of $\mathcal{C}_{1}^{\prime}$, which contains
$\mathcal{C}_{2}$ in its upper part. In particular, $\left(  P_{\alpha
}\right)  $ has exactly two solutions in $\mathcal{P}^{\circ}$ for all
$\alpha\in(0,\alpha_{s})$, see Figure \ref{fig19_0422compo}(ii).

\item Assume in addition that $\partial\Omega\subseteq\partial\Omega_{+}^{a}$
and $q\in\mathcal{A}_{\mathcal{N}}$. Then the nontrivial solutions set of
$\left(  P_{\alpha}\right)  $ with $\alpha>0$ is given exactly by
$\mathcal{C}_{1}^{\prime}$,

\end{enumerate}
\end{theorem}

Let us note that several examples of weights $a$ satisfying (\ref{hip}) are
discussed in Section \ref{sec:exa}.

\begin{rem}
\label{r1} \strut

\begin{enumerate}
\item As one can see from its proof, Theorem \ref{ball} holds more generally
under the condition $\alpha_{s}<\alpha_{2}$, instead of \eqref{hip}.

\item In Theorems \ref{mthm} and \ref{ball}, if we assume $q\in\mathcal{A}%
_{\mathcal{N}}$ instead of $q\in\mathcal{I}_{\mathcal{N}}$, then $\left(
A.1\right)  $ is no longer needed, cf. Remark \ref{lbuNN}.
\end{enumerate}
\end{rem}

To conclude the study of $(P_{\alpha})$, we carry out a bifurcation analysis,
with $q$ as parameter, and obtain solutions in $\mathcal{P}^{\circ}$ for
$\alpha>0$ fixed and $q\rightarrow1^{-}$. As a consequence, we provide the
limiting behavior as $q\rightarrow1^{-}$ of $\alpha_{s}(a,q)$ and
$u_{i,\alpha}(q)$, $i=1,2$, whenever these are the only solutions in
$\mathcal{P}^{\circ}$, see Theorem \ref{tt2} and Corollary \ref{cor:asympt}.

Finally, let us mention that $(P_{\alpha})$ has many similarities with the
Neumann problem
\[%
\begin{cases}
-\Delta u=\alpha u+a(x)u^{q} & \mbox{in $\Omega$},\\
u\geq0 & \mbox{in $\Omega$},\\
\partial_{\nu}u=0 & \mbox{on $\partial \Omega$}.
\end{cases}
\leqno{(S_\alpha)}
\]
As a matter of fact, one may see that the linear term $\alpha u$ in the
equation above acts very similarly as in the boundary condition of
$(P_{\alpha})$, as far as the methods used in this article (and in
\cite{KRQU2019}) are concerned. In this way, we complement the results
established for $(S_{\alpha})$ by Alama in \cite{alama}.

The rest of this article is organized as follows. In the next section we give
some examples of $a$ and $q$ satisfying our assumptions. Theorems \ref{tp1}
and \ref{mthm} are proved in Section \ref{sec:pr1-2}, whereas Theorem
\ref{ball}
is proved in Section \ref{sec:prth1.3}. In Section \ref{sec:bif} we follow a
bifurcation approach for $(P_{\alpha})$ with $q\rightarrow1^{-}$. Finally, in
Section \ref{sec:S} we state some results on $(S_{\alpha})$. We also include
an Appendix containing some necessary facts related to some of the stability
results in \cite[Chapter II]{Am76} .

\section{Examples}

\label{sec:exa}

Before proving our results we discuss some examples of $a$ and $q$ satisfying
our assumptions.

\subsection{On the condition \eqref{hip}}

Let $\sigma_{1}^{\mathcal{D}}(a)$ and $\sigma_{1}^{\mathcal{N}}(a)$ be the
unique positive principal eigenvalues with respect to the weight $a$, under
homogeneous Dirichlet and Neumann boundary conditions, respectively. First
note that, for $q$ close to $1$, by \cite[Theorem 1.2]{KRQUnodea}, we have
\begin{equation}
u_{\mathcal{N}}^{1-q}\simeq\frac{\left(  t_{\ast}\phi_{1}\right)  ^{1-q}%
}{\sigma_{1}^{\mathcal{N}}(a)}, \label{eun}%
\end{equation}
where $\phi_{1}\in\mathcal{P}^{\circ}$ is the eigenfunction associated to
$\sigma_{1}^{\mathcal{N}}(a)$, normalized by $\left\Vert \phi_{1}\right\Vert
_{2}=1$,
and
\[
t_{\ast}:=\exp\left[  -\frac{\int_{\Omega}a\left(  x\right)  \phi_{1}^{2}%
\log\phi_{1}}{\int_{\Omega}a\left(  x\right)  \phi_{1}^{2}}\right]  .
\]
In particular, $t_{\ast}$ and $\phi_{1}$ do not depend on $q$, and both are
away from $0$. So, for $q$ close enough to $1$ we have that $u_{\mathcal{N}%
}^{1-q}\simeq\frac{1}{\sigma_{1}^{\mathcal{N}}(a)}$. Thus, for such $q$,
\eqref{hip} is satisfied if
\begin{equation}
-\sigma_{1}^{\mathcal{N}}(a)\int_{\Omega}a<\left\vert \partial\Omega
\right\vert \alpha_{2}. \label{mu}%
\end{equation}
Recall also that $q\in\mathcal{A}_{\mathcal{N}}\subseteq\mathcal{I}%
_{\mathcal{N}}$ whenever $q$ is sufficiently close to $1$. Let us show now
examples of $a$ satisfying \eqref{mu}:

\begin{enumerate}
\item Choose any $a\in C^{\theta}\left(  \overline{\Omega}\right)  $
satisfying $\left(  A.0\right)  $ and $\left(  A.1\right)  $. Taking into
account $\left(  A.0\right)  $, we may fix $k$ such that
\begin{equation}
\max\left\{  C_{a}-\frac{\left\vert \partial\Omega\right\vert \alpha_{2}%
}{\sigma_{1}^{\mathcal{N}}\left(  a\right)  \int_{\Omega}a^{+}},1\right\}
<k<C_{a},\quad\text{with\quad}C_{a}:=\frac{\int_{\Omega}a^{-}}{\int_{\Omega
}a^{+}}, \label{k}%
\end{equation}
where, as usual, $a^{\pm}:=\max\left\{  \pm a,0\right\}  $. We
set $a_{k}:=ka^{+}-a^{-}$. Then $\left(  A.1\right)  $ holds for $a_{k}$, and
the second inequality in (\ref{k}) shows that $\left(  A.0\right)  $ is also
true for $a_{k}$. Moreover, the first inequality in (\ref{k}) gives that
$-\sigma_{1}^{\mathcal{N}}\left(  a\right)  \int_{\Omega}a_{k}<\left\vert
\partial\Omega\right\vert \alpha_{2}$, and so, since $k\geq1$, by the
monotonicity of the principal eigenvalue with respect to the weight (e.g.
\cite[Proposition 4.7]{HT}) we deduce that $a_{k}$ satisfies \eqref{mu}.

\item One can argue similarly if $a\in C(\overline{\Omega})$ is such that
$\int_{\Omega}a=0$ and $\left\{  a_{j}\right\}  \subset C^{\theta}%
(\overline{\Omega})$ is a sequence
satisfying $(A.0)$, $(A.1)$, and $a_{j}\rightarrow a$ in $L^{1}(\Omega)$.
Then, since $\sigma_{1}^{\mathcal{N}}(a)=0$ and $a\mapsto\sigma_{1}%
^{\mathcal{N}}(a)$ is continuous from $L^{1}\left(  \Omega\right)  $ into
$\mathbb{R}$ (cf. \cite[Lemma 2.4]{LY}, which also holds if $\int_{\Omega}%
a=0$), we see that (\ref{mu}) is fulfilled by $a_{j}$ with $j$ large enough.

\item Another class of weights $a$ satisfying \eqref{mu} is the following one:
denote by $B_{R}$ the ball of radius $R$ centered at $0$ in $\mathbb{R}^{N}$.
Let $a\in C^{\theta}(\overline{B_{1}})$ satisfy $(A.0)$ and $(A.1)$, and
$a_{R}\in C^{\theta}(\overline{B_{R}})$ be an extension of $a$ to $B_{R}$ such
that $(A.1)$ holds and $\int_{B_{1}}a<\int_{B_{R}}a_{R}<0$. We claim that for
$N\geq3$ and $R$ large enough $a_{R}$ fulfills \eqref{mu}. To this end, note
that $\sigma_{1}^{\mathcal{N}}(a_{R})<\sigma_{1}^{\mathcal{D}}(a)$ for every
$R>1$. Recall also that $|\partial B_{R}|=NR^{N-1}\omega_{N}$, where
$\omega_{N}$ is the volume of $B_{1}$. Thus%
\[
-\sigma_{1}^{\mathcal{N}}(a_{R})\int_{B_{R}}a_{R}<-\sigma_{1}^{\mathcal{D}%
}(a)\int_{B_{1}}a<NR^{N-2}\omega_{N}=|\partial B_{R}|\alpha_{2}(B_{R})
\]
if $N\geq3$ and $R$ is large enough, which yields \eqref{mu}. \newline
\end{enumerate}

We next modify the construction of the preceeding examples and we prove
\eqref{hip} directly, not by way of \eqref{mu}.

\begin{enumerate}
\item[(iv)] Let $a,$ $a_{j}$ be as in (ii), and satisfying in addition that
$a_{j}\rightarrow a$ in $L^{\infty}(\Omega)$. Choose $0<\varepsilon<1$ and set
$b_{\varepsilon}:=\left(  1-\varepsilon\right)  a^{+}-a^{-}$. Then
$b_{\varepsilon}$ changes sign in $\Omega$ and $b_{\varepsilon}$ fulfills
(A.0). Let us show that, for any $q\in\mathcal{I}_{\mathcal{N}}(b_{\varepsilon
})$, $a_{j}$ satisfies (\ref{mu}) for all $j$ large enough. Indeed, let
$q\in\mathcal{I}_{\mathcal{N}}(b_{\varepsilon})$, and let us stress the
dependence of $(P_{0})$ on the weight $a$ by writing $(P_{0,a})$. We note
first that $b_{\varepsilon}\leq a_{j}$ if $j$ is large enough, and thus
$u_{\mathcal{N}}(b_{\varepsilon})$ is a subsolution of $(P_{0,a_{j}})$ for
such $j$. Now, by \cite[Lemma 2.4]{BPT2}, this problem has arbitrarily large
supersolutions (because $a_{j}$ satisfies (A.0)), so we deduce that
$u_{\mathcal{N}}(a_{j})$ exists and $u_{\mathcal{N}}(a_{j})\geq u_{\mathcal{N}%
}(b_{\varepsilon})$. In particular, $q\in\mathcal{I}_{\mathcal{N}}\left(
a_{j}\right)  $. Hence, since $\int_{\Omega}a_{j}\rightarrow0$, enlarging $j$
if necessary,%
\[
\frac{-\int_{\Omega}a_{j}}{\int_{\partial\Omega}u_{\mathcal{N}}(a_{j})^{1-q}%
}\leq\frac{-\int_{\Omega}a_{j}}{\int_{\partial\Omega}u_{\mathcal{N}%
}(b_{\varepsilon})^{1-q}}\leq\alpha_{2},
\]
as claimed.

\item[(v)] As a particular case of (iv) we see that $a_{k}$ in (i) satisfies
\eqref{hip} also for any $q\in\mathcal{I}_{\mathcal{N}}(a)$ and $k\approx
C_{a}$. Indeed, for $a$ as in (i), define $\widetilde{a}:=C_{a}a^{+}-a^{-}$,
$\varepsilon:=1-1/C_{a}$ and $\widetilde{b}_{\varepsilon}:=\left(
1-\varepsilon\right)  C_{a}a^{+}-a^{-}$. Note that $\int_{\Omega}\widetilde
{a}=0$ and $\widetilde{b}_{\varepsilon}=a$. Also, for any $\left\{
k_{j}\right\}  $ with $0<k_{j}\nearrow C_{a}$, let $\widetilde{a}_{j}%
:=k_{j}a^{+}-a^{-}$. Then, item (iv) with $\widetilde{a}$, $\widetilde{a}_{j}$
and $\widetilde{b}_{\varepsilon}$ in place of $a$, $a_{j}$ and $b_{\varepsilon
}$, respectively, yields the desired assertion.
\end{enumerate}

\subsection{On the condition $q \in\mathcal{A}_{\mathcal{N}}$}

As already mentioned, under $\left(  A.0\right)  $ and $(A.1)$, we have
$\mathcal{A}_{\mathcal{N}}=\left(  q_{\mathcal{N}},1\right)  $ for some
$q_{\mathcal{N}}=q_{\mathcal{N}}(a)\geq0$. However, a useful \textit{upper}
estimate
(i.e. $<1$) on $q_{\mathcal{N}}$ is hard to obtain in general, since
$q_{\mathcal{N}}$ can be arbitrarily close
to $1$ (see \cite[Theorem 1.4(ii)]{KRQUnodea}). Let us show a situation where
this can be done. Take $0<R_{0}<R$, $\Omega:=B_{R}$, and $a\in C^{\theta
}\left(  \overline{\Omega}\right)  $ a radial function satisfying $\left(
A.0\right)  $, $a\leq0$ in $\overline{B_{R_{0}}}$ and $a>0$ in $A_{R_{0}%
}:=\Omega\diagdown\overline{B_{R_{0}}}$ (note that $a$ fulfills $\left(
A.1\right)  $ and $\partial\Omega\subseteq\partial\Omega_{+}^{a}$). Then, as a
consequence of \cite[Corollary 4.4]{KRQUnodea} we have that
\[
q_{\mathcal{N}}\leq\frac{1-K}{1-K+2KN^{-1}}:=\underline{q},\quad
\text{where\quad}K=K\left(  a\right)  :=\frac{\int_{A_{R_{0}}}a^{+}%
}{\left\vert B_{R_{0}}\right\vert \left\Vert a^{-}\right\Vert _{L^{\infty
}(B_{R_{0}})}}.
\]
Let us observe that, by $\left(  A.0\right)  $, $K<1$. A similar result holds
in the one-dimensional case without requiring any eveness assumptions, see
\cite[Remark 3.4]{KRQUnodea}.

Also, an inspection of the proof of \cite[Theorem 1.8 and Corollary
4.4]{KRQUnodea} shows that, for $a$ as above and $q\in(\underline{q},1)$, we
have
\[
u_{\mathcal{N}}^{1-q}\geq\frac{1-q}{\left\vert \partial B_{R}\right\vert }%
\int_{R_{0}}^{R}\left(  \int_{A_{t}}a^{+}\right)  dt\quad\mbox{on }\partial
B_{R}.
\]
Thus, \eqref{hip} holds whenever%
\[
\frac{\int_{B_{R_{0}}}a^{-}-\int_{A_{R_{0}}}a^{+}}{\int_{R_{0}}^{R}%
(\int_{A_{t}}a^{+})dt}\leq\frac{1-q}{R}.
\]
In particular, it follows that given \textit{any} $q\in\left(  0,1\right)  $,
we can find a weight $a$ such that $q\in\mathcal{A}_{\mathcal{N}}(a)$ and
\eqref{hip} holds, choosing suitably $a$ so that $\int_{B_{R}}a\approx0$ and
$K\approx1$.

\section{Proof of Theorems \ref{tp1} and \ref{mthm}}

\label{sec:pr1-2}

\subsection{Proof of Theorem \ref{tp1}}

The proof of Theorem \ref{tp1} is a direct consequence of Propositions
\ref{p1} and \ref{lbuN}(i) below.

\begin{prop}
\label{p1} Let $\alpha\geq0$ and $0<q<1$. If $u$ is a supersolution of
$(P_{\alpha})$ such that $u>0$ in $\Omega_{+}^{a}$ then $\int_{\text{supp }%
u}a<0$. In particular $(A.0)$ holds.
\end{prop}

\noindent\textit{Proof}. We argue as in \cite[Proposition 2.3]{alama}, which
is inspired by \cite[Lemma 2.1]{BPT2}. Take $(u+\varepsilon)^{-q}$ as test
function in $(P_{\alpha})$ to get
\[
-q\int_{\Omega}\frac{\left\vert \nabla u\right\vert ^{2}}{(u+\varepsilon
)^{q+1}}\geq\int_{\Omega}a\left(  \frac{u}{u+\varepsilon}\right)  ^{q}%
+\alpha\int_{\partial\Omega}\frac{u}{(u+\varepsilon)^{q}}\geq\int_{\Omega
}a\left(  \frac{u}{u+\varepsilon}\right)  ^{q}.
\]
Following the argument in \cite[Proposition 2.3]{alama} we deduce that
$\int_{\text{supp }u}a<0$. Since $a\leq0$ on the region where $u$ possibly
vanishes, we find that $\int_{\Omega}a\leq\int_{\text{supp }u}a<0$. \qed

\begin{prop}
\label{lbuN}

Assume $(A.0)$, $(A.1)$ and $q \in\mathcal{I}_{\mathcal{N}}$.

\begin{enumerate}
\item If $u$ is a supersolution of $(P_{\alpha})$ for $\alpha>0$ and $u>0$ in
$\Omega_{+}^{a}$, then $u\geq u_{\mathcal{N}}$.

\item For $\alpha>0$ small enough $(P_{\alpha})$ has, in $\mathcal{P}^{\circ}%
$, exactly two solutions $u_{1,\alpha},u_{2,\alpha}$. Moreover:

\begin{enumerate}
\item $u_{1,\alpha}=u_{\mathcal{N}}$ at $\alpha=0$, and there exists some
$\alpha_{\ast}>0$ such that the map $\alpha\mapsto u_{1,\alpha}$ is
$C^{\infty}$ from $[0,\alpha_{\ast})$ to $C^{1}(\overline{\Omega})$,
increasing, $\gamma_{1}(\alpha,u_{1,\alpha})>0$, and $\gamma_{1}(\alpha_{\ast
},u_{1,\alpha_{\ast}})=0$. Moreover, for every $\beta\in(0,\alpha_{s})$ the
solutions set of $(P_{\alpha})$ in a neighborhood of $(\beta,u_{1,\beta})$ is
given exactly by $(\alpha,u_{1,\alpha})$.

\item there exists some $\overline{\alpha}>0$ such that the map $\alpha\mapsto
u_{2,\alpha}$ is $C^{\infty}$ from $(0,\overline{\alpha})$ to $C^{1}%
(\overline{\Omega})$, and $\min_{\overline{\Omega}}u_{2,\alpha}\rightarrow
\infty$ as $\alpha\rightarrow0^{+}$. Moreover, for every $\beta\in
(0,\overline{\alpha})$ the solutions set in a neighborhood of $(\beta
,u_{1,\beta})$ is given exactly by $(\alpha,u_{2,\alpha})$.
\end{enumerate}
\end{enumerate}
\end{prop}

\noindent\textit{Proof}.

\begin{enumerate}
\item First of all, it is clear that $u$ is a supersolution of $(P_{0})$.
Now, by $\left(  A.1\right)  $ we have that $\Omega_{+}^{a}=\cup_{k=1}%
^{n}\Omega_{k}$, where $\Omega_{k}$ are smooth, open and connected. For each
$k$, we take a ball $B_{k}\Subset\Omega_{k}$
and consider the Dirichlet eigenvalue problem
\[%
\begin{cases}
-\Delta\phi=\lambda a(x)\phi & \mbox{ in }B_{k},\\
\phi=0 & \mbox{ on }\partial B_{k}.
\end{cases}
\]
By $\phi_{k}$ we denote a positive eigenfunction associated with the first
eigenvalue of this problem, and extended by $0$ outside $B_{k}$. If
$\varepsilon=\varepsilon_{k}>0$ is small enough then $\varepsilon\phi_{k}$ is
a subsolution of $(P_{0})$ smaller than $u$, since $u>0$ in every $\Omega_{k}%
$. Hence $\hat{\phi}:=\max_{k}\{\varepsilon_{k}\phi_{k}\}$ is a (weak)
subsolution of $(P_{0})$ and $\hat{\phi}>0$ in $B_{k}$ for all $k$. By the sub
and supersolutions method, we have a solution $v$ of $(P_{0})$ such that
$\hat{\phi}\leq v\leq u$. In particular, $v>0$ in $B_{k}$ for all $k$. By the
strong maximum principle, we deduce that $v>0$ in $\Omega_{k}$ for all $k$,
i.e. $v>0$ in $\Omega_{+}^{a}$. Now, under $\left(  A.1\right)  $, there is at
most one solution of $\left(  P_{0}\right)  $ which is positive in $\Omega
_{+}^{a}$ \cite[Theorem 3.1]{BPT2}. Therefore, $v=u_{\mathcal{N}}$, and so
$u_{\mathcal{N}}\leq u$, as claimed.

\item Since $\gamma_{1}(0,u_{\mathcal{N}})>0$ (see \cite[Lemma 2.5]%
{KRQU2019}), the IFT provides the existence of a $C^{\infty}$ (and increasing,
by \cite[Theorem 7.10]{LGbook13}) solution curve $\alpha\mapsto u_{1,\alpha
}\in\mathcal{P}^{\circ}$ for $\alpha>0$ small. Let
$\alpha_{\ast}>0$ be the maximal $\alpha$ of this curve. Recalling that
$\alpha_{s}<\infty$ we have that $\alpha_{\ast}<\infty$. Furthermore, taking
into account the fact that $u_{1,\alpha}$ is increasing and the \textit{a
priori} upper bound of nontrivial solutions of $(P_{\alpha})$ for a compact
interval in $(0,\infty)$ (see \cite[Proposition 3.2]{KRQU2019}), letting
$\alpha\nearrow\alpha_{\ast}$ we see that there exists $u_{\alpha_{\ast}}%
\in\mathcal{P}^{\circ}$ solution of $\left(  P_{\alpha_{\ast}}\right)  $.
Moreover $\gamma_{1}(\alpha_{\ast},u_{1,\alpha_{\ast}})=0$ by the maximality
of $\alpha_{\ast}$ (otherwise we use the IFT to extend the curve beyond
$\alpha_{\ast}$) . \newline On the other hand, let $w(\alpha):=t(\alpha
)+\psi(\alpha,t(\alpha))$, where $t(\alpha)$ and $\psi(\alpha,t(\alpha))$ are
as in \cite[Proposition 3.11]{KRQU2019}. By this proposition, for all
$\alpha>0$ small enough $w(\alpha)\in\mathcal{P}^{\circ}$ is a solution of the
problem
\[%
\begin{cases}
-\Delta w=\alpha a(x)w^{q} & \mbox{in}\ \Omega,\\
w\geq0 & \mbox{in}\ \Omega,\\
\partial_{\nu}w=\alpha w & \mbox{on}\ \partial\Omega,
\end{cases}
\leqno{(R_\alpha)}
\]
satisfying $t(0)=c_{a}$ and $\psi(0,c_{a})=0$, where $c_{a}$ is the positive
constant introduced by \eqref{def:ca}.
In particular, $w(0)=c_{a}>0$, and still by \cite[Proposition 3.11]{KRQU2019}
we know that $\alpha\mapsto w(\alpha)$ is a $C^{\infty}$ mapping from
$(-\alpha_{0},\alpha_{0})$ to $W^{2,s}(\Omega)$, $s>N$,
for some $\alpha_{0}>0$ small. Moreover, it is easy to see that $u_{2,\alpha
}:=\alpha^{-\frac{1}{1-q}}w(\alpha)$ is a solution of $\left(  P_{\alpha
}\right)  $ with $\min_{\overline{\Omega}}u_{2,\alpha}\rightarrow\infty$ as
$\alpha\rightarrow0^{+}$.\newline Let us show the exactness assertion. From
\cite[Propositions 3.2, 3.10, 3.11(ii)]{KRQU2019} and the change of variables
$u=\alpha^{-\frac{1}{1-q}}w$, we deduce that $u_{2,\alpha}$ is the only
solution which blows up as $\alpha\rightarrow0^{+}$. More precisely, there
exists $C>0$ such that, except for $u_{2,\alpha}$, we have $\Vert
u\Vert_{C(\overline{\Omega})}\leq C$ for all nontrivial solutions of
$(P_{\alpha})$ for $\alpha$ close to $0$. Assume $u_{j}\in\mathcal{P}^{\circ}$
is a solution of $(P_{\alpha_{j}})$ with $\alpha_{j}\rightarrow0^{+}$ and
$u_{j}$ is not on the curve $(\alpha,u_{2,\alpha})$. By (i) we know that
$u_{\mathcal{N}}\leq u_{j}\leq C$ on $\overline{\Omega}$. Up to a subsequence,
we have $u_{j}\rightarrow u_{0}$, $u_{0}\geq u_{\mathcal{N}}$, and $u_{0}$ is
a solution of $(P_{0})$. Since $u_{\mathcal{N}}$ is unique we get
$u_{0}=u_{\mathcal{N}}$, and by the IFT, we conclude that $u_{j}$ is on the
curve $(\alpha,u_{1,\alpha})$ for $j$ large enough. The remaining assertions
follow directly from the IFT. \qed\newline
\end{enumerate}

\begin{rem}
\label{lbuNN} As an alternative to Proposition \ref{lbuN}(i), one can show
that if $q\in\mathcal{A}_{\mathcal{N}}$ and $u$ is a supersolution of
$(P_{\alpha})$ with $\alpha>0$ and $u \not \equiv 0$ in $\Omega_{+}^{a}$ then
$u\geq u_{\mathcal{N}}$. Indeed, arguing as in the proof of Proposition
\ref{lbuN}(i) we may construct a nontrivial solution $v$ of $\left(
P_{0}\right)  $ with $v\leq u$. Since $q\in\mathcal{A}_{\mathcal{N}}$ we must
have $v\in\mathcal{P}^{\circ}$. Using the uniqueness result in \cite[Lemma
3.1]{BPT2} we deduce that $v=u_{\mathcal{N}}$, and the conclusion follows.
\end{rem}

\subsection{Proof of Theorem \ref{mthm}}

Theorem \ref{mthm} is an immediate consequence of Theorem \ref{mthm2} below.
Before stating this result, we recall in the next paragraph some facts
concerning the stability of the solutions $u\in\mathcal{P}^{\circ}$ of
$(P_{\alpha})$.

Given a solution $u\in\mathcal{P}^{\circ}$ of $(P_{\alpha})$, let us consider
the linearized eigenvalue problem
\begin{equation}%
\begin{cases}
\mathcal{L}(\alpha,u)\phi:=-\Delta\phi-qa(x)u^{q-1}\phi=\gamma(\alpha,u)\phi &
\mbox{in $\Omega$},\\
\mathcal{B}(\alpha,u)\phi:=\partial_{\nu}\phi-\alpha\phi=0 &
\mbox{on $\partial \Omega$}.
\end{cases}
\label{ep}%
\end{equation}
We denote by $\gamma_{1}=\gamma_{1}(\alpha,u)$ the smallest eigenvalue of
\eqref{ep} and by $\phi_{1}=\phi_{1}(\alpha,u)$ a positive eigenfunction
associated with $\gamma_{1}$.
We remark that $\gamma_{1}$ is simple, and $\phi_{1}\in\mathcal{P}^{\circ}$.
Also, if $\alpha\mapsto u(\alpha)$ is continuous in $C^{2+r}(\overline{\Omega
})$ for some $r\in(0,1)$, then so is $\alpha\mapsto\gamma_{1}(\alpha
,u(\alpha))$. In fact, with computations as those in \cite[p. 1155]{U12} one
can see that this map inherits the regularity of the map $\alpha\mapsto
u=u(\alpha)$, by the implicit function theorem (in short, IFT). Recall that
$u\in\mathcal{P}^{\circ}$ is said to be

\begin{itemize}
\item \textit{asymptotically stable} if $\gamma_{1}(\alpha,u)>0$,

\item \textit{weakly stable} if $\gamma_{1}(\alpha,u)\geq0$,

\item \textit{unstable} if $\gamma_{1}(\alpha,u)<0$.
\end{itemize}


\begin{theorem}
\label{mthm2}

Assume $(A.0),$ $(A.1)$, and $q\in\mathcal{I}_{\mathcal{N}}$. Then:

\begin{enumerate}
\item $(P_{\alpha})$ has a solution curve $\mathcal{C}_{1}=\left\{
(\alpha,u_{1,\alpha});0\leq\alpha\leq\alpha_{s}\right\}  $ such that
$\alpha\mapsto u_{1,\alpha}\in\mathcal{P}^{\circ}$ is continuous and
increasing on $[0,\alpha_{s}]$ and $C^{\infty}$ in $[0,\alpha_{s})$, and
$u_{1,0}=u_{\mathcal{N}}$. Moreover,

\begin{enumerate}
\item $u_{1,\alpha_{s}}$ is the unique solution of $(P_{\alpha_{s}})$ in
$\mathcal{P}^{\circ}$;

\item $u_{1,\alpha}$ is minimal in $\mathcal{P}^{\circ}$ for $\alpha\in
\lbrack0,\alpha_{s})$;

\item for every $\beta\in(0,\alpha_{s})$ the solutions set of $(P_{\alpha})$
in a neighborhood of $(\beta,u_{1,\beta})$ is given exactly by $\mathcal{C}%
_{1}$.

\item $u_{1,\alpha}$ is asymptotically stable for $\alpha\in\lbrack
0,\alpha_{s})$, and weakly stable but not asymptotically stable for
$\alpha=\alpha_{s}$.
Furthermore, $u_{1,\alpha}$ is the only stable solution of $(P_{\alpha})$ for
$\alpha\in(0,\alpha_{s}]$.
\end{enumerate}

\item $(P_{\alpha})$ has a solution curve $\mathcal{C}_{2}=\left\{
(\alpha,u_{2,\alpha});0<\alpha\leq\overline{\alpha}\right\}  $ for some
$\overline{\alpha}\in(0,\alpha_{s}]$, such that $\alpha\mapsto u_{2,\alpha}%
\in\mathcal{P}^{\circ}$ is continuous and decreasing on
$(0,\overline{\alpha}]$ and $C^{\infty}$ in $(0,\overline{\alpha})$. In
addition, $u_{2,\alpha}$ is unstable for $\alpha\in(0,\overline{\alpha})$ and
$\min_{\overline{\Omega}}u_{2,\alpha}\rightarrow\infty$ as $\alpha
\rightarrow0^{+}$. Moreover:

\begin{enumerate}
\item for every $\beta\in(0,\overline{\alpha})$ the solutions set of
$(P_{\alpha})$ in a neighborhood of $(\beta,u_{2,\beta})$ is given exactly by
$\mathcal{C}_{2}$;

\item for every $\alpha\in(0,\overline{\alpha})$ the solutions $u_{1,\alpha
},u_{2,\alpha}$ are strictly ordered as follows: $u_{2,\alpha}-u_{1,\alpha}%
\in\mathcal{P}^{\circ}$;

\item $u_{1,\alpha}$ and $u_{2,\alpha}$ are the only solutions of $(P_{\alpha
})$ in $\mathcal{P}^{\circ}$ for $\alpha>0$ small.
\end{enumerate}

\item $(P_{\alpha})$
has a $C^{\infty}$ solution curve $\mathcal{C}_{3}=\{(\alpha(t),u(t));-t_{0}%
<t<t_{0}\}$ for some $t_{0}>0$ small, such that $u(t)\in\mathcal{P}^{\circ}$,
$(\alpha(0),u(0))=(\alpha_{s},u_{1,\alpha_{s}})$, $\alpha^{\prime}%
(0)=0>\alpha^{\prime\prime}(0)$ and
$u^{\prime}(0)\in\mathcal{P}^{\circ}$, i.e., the solution curve \textrm{bends
to the left} in a neighborhood of $(\alpha_{s},u_{1,\alpha_{s}})$. In addition:

\begin{enumerate}
\item $\{(\alpha(t),u(t));-t_{0}<t<0\}$ represents the lower branch of
$\mathcal{C}_{3}$ satisfying $u(t)=u_{1,\alpha(t)}$, whereas $\{(\alpha
(t),u(t));0<t<t_{0}\}$ represents the upper one; these branches are increasing
and decreasing, respectively.

\item in a neighborhood of $(\alpha_{s},u_{1,\alpha_{s}})$ the solutions set
of $(P_{\alpha})$ is given exactly by $\mathcal{C}_{3}$.
\end{enumerate}

Furthermore, $(P_{\alpha})$ has exactly two solutions in $\mathcal{P}^{\circ}$
for $\alpha$ in a left neighborhood of $\alpha_{s}$.

\item Assume additionally that $0\not \equiv a\geq0$ in some smooth domain
$D\subset\Omega$ such that $\left\vert \partial D\cap\partial\Omega\right\vert
>0$. Then $(P_{\alpha})$ has a \textrm{component} \textrm{(}i.e., a maximal
closed, connected subset\textrm{)} $\mathcal{C}_{\ast}$ of solutions in
$[0,\alpha_{s}]\times\mathcal{P}^{\circ}$ which includes $\mathcal{C}_{1}$,
$\mathcal{C}_{2}$ and $\mathcal{C}_{3}$, and satisfies
\begin{equation}
\mathcal{C}_{\ast}\cap\{(\alpha,0),(\alpha,\infty),(0,u)\}=\{(0,\infty
),(0,u_{\mathcal{N}})\}. \label{compoCast}%
\end{equation}
In particular, $(P_{\alpha})$ has at least two solutions in $\mathcal{P}%
^{\circ}$ for every $\alpha\in(0,\alpha_{s})$, see
Figure \ref{fig19_0422compo}(i).
\end{enumerate}
\end{theorem}


The rest of this subsection is devoted to the proof of the above theorem. In
order to prove it, we start with two further results on $u_{2,\alpha}$.

\begin{prop}
\label{monoton}Assume $\left(  A.0\right)  $. Then:

\begin{enumerate}
\item $u_{2,\alpha}>u_{2,\beta}$ on $\overline{\Omega}$ if $0<\alpha<\beta$
are small enough.

\item $u_{2,\alpha}$ is unstable if $\alpha>0$ is small enough.
\end{enumerate}
\end{prop}

\noindent\textit{Proof}. \strut

\begin{enumerate}
\item Let $w\left(  \alpha\right)  $ be as in the proof of Proposition
\ref{lbuN}(ii). Since $u_{2,\alpha}=\alpha^{-\frac{1}{1-q}}w(\alpha)$,
differentiating $u_{2,\alpha}$ with respect to $\alpha$ provides
\[
\frac{du_{2,\alpha}}{d\alpha}=-\frac{1}{1-q}\alpha^{-\frac{1}{1-q}-1}%
w(\alpha)+\alpha^{-\frac{1}{1-q}}w^{\prime}(\alpha).
\]
Set
\begin{equation}
\eta(\alpha):=\alpha^{\frac{2-q}{1-q}}\frac{du_{2,\alpha}}{d\alpha}=-\frac
{1}{1-q}w(\alpha)+\alpha w^{\prime}(\alpha). \label{eta}%
\end{equation}
Recalling that $w\left(  0\right)  =c_{a}$ we see that%
\begin{equation}
\eta(0)=-\frac{1}{1-q}w(0)=-\frac{c_{a}}{1-q}<0. \label{eta0}%
\end{equation}
Moreover, we know that $w(\alpha)\rightarrow w(0)$ in $C(\overline{\Omega})$
as $\alpha\rightarrow0^{+}$, and also that $\Vert w^{\prime}(\alpha
)\Vert_{C(\overline{\Omega})}$ is bounded as $\alpha\rightarrow0^{+}$, since
$w^{\prime}(\alpha)\rightarrow w^{\prime}(0)$ in $C(\overline{\Omega})$ as
$\alpha\rightarrow0^{+}$. Hence, we deduce from \eqref{eta} and \eqref{eta0}
that
\[
\eta(\alpha)\rightarrow\eta(0)=-\frac{c_{a}}{1-q}<0\quad\mbox{in}\ C(\overline
{\Omega})\ \mbox{as}\ \alpha\rightarrow0^{+},
\]
and then,
\[
\alpha^{\frac{2-q}{1-q}}\frac{du_{2,\alpha}}{d\alpha}\longrightarrow
-\frac{c_{a}}{1-q}\quad\mbox{ in}\ C(\overline{\Omega})\ \mbox{ as }\ \alpha
\rightarrow0^{+}.
\]
In particular, there exists $\underline{\alpha}>0$ such that
\[
\frac{du_{2,\alpha}}{d\alpha}<-\frac{c_{a}}{2(1-q)}\alpha^{-\frac{2-q}{1-q}%
}\quad\mbox{ on }\ \overline{\Omega}\ \mbox{ for }\ 0<\alpha<\underline
{\alpha},
\]
which
yields the conclusion.

\item As mentioned in the proof of Proposition \ref{lbuN}(ii),
\begin{equation}
\alpha^{\frac{1}{1-q}}u_{2,\alpha}=w\left(  \alpha\right)  \rightarrow
c_{a}>0\quad\mbox{in}\quad C^{1}(\overline{\Omega}),\quad\mbox{as}\quad
\alpha\rightarrow0^{+}. \label{convca}%
\end{equation}
Taking $\frac{1}{\phi_{1}}$ as test function in \eqref{ep} we see that
\[
0>\int_{\Omega}\nabla\phi_{1}\nabla\left(  \frac{1}{\phi_{1}}\right)
=\alpha|\partial\Omega|+\gamma_{1}|\Omega|+q\int_{\Omega}au_{2,\alpha}^{q-1},
\]
so that
\[
-\frac{\gamma_{1}|\Omega|}{\alpha}>q\int_{\Omega}a\left(  \alpha^{\frac
{1}{1-q}}u_{2,\alpha}\right)  ^{q-1}+|\partial\Omega|,
\]
and \eqref{convca} provides
\[
\limsup_{\alpha\rightarrow0^{+}}\frac{\gamma_{1}}{\alpha}\leq-\frac
{(1-q)\left\vert \partial\Omega\right\vert }{|\Omega|}<0.
\]
The desired conclusion follows. \qed \newline
\end{enumerate}

We shall employ the following result to study the solutions set of
$(P_{\alpha})$ in a neighborhood of $(\alpha_{\ast},u_{1,\alpha_{\ast}})$,
where $\alpha_{\ast}$ is as in Proposition
\ref{lbuN}(ii-a).

\begin{prop}
\label{prop:bend}

Let $u_{0}\in\mathcal{P}^{\circ}$ be a solution of $(P_{\alpha_{0}})$ for
$\alpha_{0}>0$
such that $\gamma_{1}(\alpha_{0},u_{0})=0$. Then, in a neighborhood of
$(\alpha_{0},u_{0})$ the solutions set of $(P_{\alpha})$ is given exactly by a
$C^{\infty}$ curve $(\alpha,u)=(\alpha(t),u(t))$, parametrized by $t\in
(-t_{0},t_{0})$ for some $t_{0}>0$, and such that $(\alpha(0),u(0))=(\alpha
_{0},u_{0})$. Moreover:

\begin{enumerate}
\item $(\alpha(t),u(t))=(\alpha_{0} + \beta(t), u_{0} + t\phi_{1} + z(t))$
with $\phi_{1}=\phi_{1}(\alpha_{0},u_{0})$ and some $C^{\infty}$ functions
$\beta(\cdot), z(\cdot)$, satisfying $\beta(0)=\beta^{\prime}(0)=0$ and
$z(0)=z^{\prime}(0)=0$
\textrm{(}implying $u^{\prime}(0)=\phi_{1}\in\mathcal{P}^{\circ}$\textrm{)};

\item $\beta^{\prime\prime}(0)<0$, i.e., the curve $(\alpha(t),u(t))$ bends to
the left in a neighborhood of $(\alpha_{0},u_{0})$, see Figure
\ref{fig19_0422compo}(i);

\item $u(t)$ is asymptotically stable for $t<0$ and unstable for $t>0$.
\end{enumerate}
\end{prop}

\noindent\textit{Proof}. We apply \cite[Theorems 3.2 and 3.6]{CR73} to prove
this proposition.

\begin{enumerate}
\item For some suitable $r=r(q,\theta)\in(0,1)$ we consider the mapping
\[
F\colon U_{0}\longrightarrow C^{r}(\overline{\Omega})\times C^{1+r}%
(\partial\Omega);\ (\alpha,u)\longmapsto(-\Delta u-au^{q},\ \partial_{\nu
}u-\alpha u),
\]
where $U_{0}$ is an open neighborhood of $(\alpha_{0},u_{0})$ in
$\mathbb{R}\times C^{2+r}(\overline{\Omega})$ such that $\alpha>0$ and
$u\in\mathcal{P}^{\circ}$ for $(\alpha,u)\in U_{0}$. We see that for a given
$(\alpha,u)\in U_{0}$, $u\in\mathcal{P}^{\circ}$ solves $(P_{\alpha})$ if and
only if $F(\alpha,u)=\left(  0,0\right)  $. The Fr\'{e}chet derivatives
$F_{\alpha}$ and $F_{u}$ are given by
\[
F_{\alpha}(\alpha_{0},u_{0})=(0,-u_{0}),\quad F_{u}(\alpha,u)\phi
=(\mathcal{L}(\alpha,u)\phi,\mathcal{B}(\alpha,u)\phi),
\]
where we recall that $\mathcal{L}$ and $\mathcal{B}$ are given by \eqref{ep}.
Since $\gamma_{1}(\alpha_{0},u_{0})=0$, we deduce that $\mathrm{Ker}%
(F_{u}(\alpha_{0},u_{0}))=\mathrm{span}\,\{\phi_{1}\}$, where $\phi_{1}%
=\phi_{1}(\alpha_{0},u_{0})\in\mathcal{P}^{\circ}$. To verify that $F_{\alpha
}(\alpha_{0},u_{0})\not \in \mathrm{Im}(F_{u}(\alpha_{0},u_{0}))$, we shall
check that the problem
\[%
\begin{cases}
\mathcal{L}(\alpha_{0},u_{0})v=0 & \mbox{in $\Omega$},\\
\mathcal{B}(\alpha_{0},u_{0})v=-u_{0} & \mbox{on $\partial \Omega$}
\end{cases}
\]
has no solution. Indeed, if $v$ solves this problem then
\begin{equation}
0=\int_{\Omega}(\phi_{1}\mathcal{L}v-v\mathcal{L}\phi_{1})=\int_{\partial
\Omega}(-\phi_{1}\mathcal{B}v+v\mathcal{B}\phi_{1})=\int_{\partial\Omega}%
u_{0}\phi_{1}>0, \label{LBgr}%
\end{equation}
a contradiction. By \cite[Theorem 3.2]{CR73} we infer that, in a neighborhood
of $(\alpha_{0},u_{0})$, the solutions of $F(\alpha,u)=\left(  0,0\right)  $
are given exactly by the curve
\begin{equation}
\{(\alpha(t),u(t))=(\alpha_{0}+\beta(t),u_{0}+t\phi_{1}+z(t)):-t_{0}<t<t_{0}\}
\label{curve}%
\end{equation}
for some $\beta(\cdot),z(\cdot)$ which are $C^{\infty}$ and satisfy
$\beta(0)=\beta^{\prime}(0)=0$, and $z(0)=z^{\prime}(0)=0$. Assertion (i) has
been verified.

\item Based on the regularity assertion for the curve, we shall differentiate
$F(\alpha(t),u(t))=\left(  0,0\right)  $ with respect to $t$ twice at $t=0$.
From \eqref{curve} we see that $\alpha^{\prime}\left(  0\right)  =0$ and
$u^{\prime}\left(  0\right)  =\phi_{1}$. Thus, after some computations we find
that
\[%
\begin{cases}
\mathcal{L}(\alpha_{0},u_{0})u^{\prime\prime}(0)=aq(q-1)u_{0}^{q-2}\phi
_{1}^{2} & \mbox{in $\Omega$},\\
\mathcal{B}(\alpha_{0},u_{0})u^{\prime\prime}(0)=\alpha^{\prime\prime}(0)u_{0}
& \mbox{on $\partial \Omega$}.
\end{cases}
\]
We now use Green's formula to deduce that
\begin{align*}
&  \int_{\Omega}(\phi_{1}\mathcal{L}(\alpha_{0},u_{0})u^{\prime\prime
}(0)-u^{\prime\prime}(0)\mathcal{L}(\alpha_{0},u_{0})\phi_{1})\\
&  =\int_{\partial\Omega}(u^{\prime\prime}(0)\mathcal{B}(\alpha_{0},u_{0}%
)\phi_{1}-\phi_{1}\mathcal{B}(\alpha_{0},u_{0})u^{\prime\prime}(0)).
\end{align*}
Since $\mathcal{L}(\alpha_{0},u_{0})\phi_{1}=0$ and $\mathcal{B}(\alpha
_{0},u_{0})\phi_{1}=0$, we find that
\begin{align*}
\int_{\Omega}(\phi_{1}\mathcal{L}(\alpha_{0},u_{0})u^{\prime\prime
}(0)-u^{\prime\prime}(0)\mathcal{L}(\alpha_{0},u_{0})\phi_{1})  &
=q(q-1)\int_{\Omega}au_{0}^{q-2}\phi_{1}^{3},\\
\int_{\partial\Omega}(u^{\prime\prime}(0)\mathcal{B}(\alpha_{0},u_{0})\phi
_{1}-\phi_{1}\mathcal{B}(\alpha_{0},u_{0})u^{\prime\prime}(0))  &
=-\alpha^{\prime\prime}(0)\int_{\partial\Omega}u_{0}\phi_{1}.
\end{align*}
Hence, we deduce that
\[
\beta^{\prime\prime}(0)=\alpha^{\prime\prime}(0)=\frac{q(1-q)\int_{\Omega
}au_{0}^{q-2}\phi_{1}^{3}}{\int_{\partial\Omega}u_{0}\phi_{1}}%
<0\ \Longleftrightarrow\ \int_{\Omega}au_{0}^{q-2}\phi_{1}^{3}<0.
\]
Let us show that $\int_{\Omega}au_{0}^{q-2}\phi_{1}^{3}<0$. Employing
$\frac{\phi_{1}^{3}}{u_{0}^{2}}$ and $\frac{\phi_{1}^{2}}{u_{0}}$ as test
functions in the weak forms of $\left(  P_{\alpha_{0}}\right)  $ and
\eqref{ep}, respectively, and recalling that $\gamma_{1}(\alpha_{0},u_{0})=0$,
we infer that%
\begin{align*}
(1-q)\int_{\Omega}au_{0}^{q-2}\phi_{1}^{3}  &  =\int_{\Omega}\nabla
u_{0}\nabla\left(  \frac{\phi_{1}^{3}}{u_{0}^{2}}\right)  -\nabla\phi
_{1}\nabla\left(  \frac{\phi_{1}^{2}}{u_{0}}\right) \\
&  =-2\int_{\Omega}\frac{\phi_{1}}{u_{0}^{3}}\left\vert u_{0}\nabla\phi
_{1}-\phi_{1}\nabla u_{0}\right\vert ^{2}<0,
\end{align*}
as desired. Assertion (ii) has been verified.

\item This assertion is a direct consequence of the previous one, using
\cite[Theorem 3.6]{CR73}. Indeed, if we define $Kw:=(w,0)$, we can check in
the same way as \eqref{LBgr} that $K\phi_{1}\not \in \mathrm{Im}(F_{u}%
(\alpha_{0},u_{0}))$, i.e., $0$ is a $K$-simple eigenvalue of $F_{u}%
(\alpha_{0},u_{0})$ in the sense of \cite[Definition 1.2]{CR73}. Put
$\gamma(t):=\gamma_{1}(\alpha(t),u(t))$ and $\phi(t):=\phi_{1}(\alpha
(t),u(t))$ to get
\[
F_{u}(\alpha(t),u(t))\phi(t)=\gamma(t)K\phi(t)\Longleftrightarrow\left\{
\begin{array}
[c]{ll}%
\mathcal{L}(\alpha(t),u(t))\phi(t)=\gamma(t)\phi(t) & \mbox{in $\Omega$},\\
\mathcal{B}(\alpha(t),u(t))\phi(t)=0 & \mbox{on $\partial \Omega$},
\end{array}
\right.
\]
see $\cite[(3.5)]{CR73}$. Lastly, using $\phi_{1}$, we set $l$ as the
continuous functional on $C^{r}(\overline{\Omega})\times C^{1+r}%
(\partial\Omega)$ given by
\[
\langle l,(f,g)\rangle:=\int_{\Omega}f\phi_{1}+\int_{\partial\Omega}g\phi
_{1},
\]
which satisfies $\mathrm{Ker}(l)=\mathrm{Im}(F_{u}(\alpha_{0},u_{0}))$ (note
that $\langle l,K\phi_{1}\rangle=\int_{\Omega}\phi_{1}^{2}\neq0$). We are now
ready to apply \cite[Theorem 3.6]{CR73} to obtain
\[
\lim_{t\rightarrow0}\frac{\gamma(t)}{\alpha^{\prime}(t)}=-\frac{\langle
l,F_{\alpha}(\alpha_{0},u_{0})\rangle}{\langle l,K\phi_{1}\rangle}
=\frac{\int_{\partial\Omega}u_{0}\phi_{1}}{\int_{\Omega}\phi_{1}^{2}}>0.
\]
This implies that $\gamma(t)$ and $\alpha^{\prime}(t)$ have the same sign when
$t$ is close to $0$, which combined with assertion (ii) provides the desired
conclusion. \qed\newline
\end{enumerate}



Now, with the aid of Proposition \ref{monoton}(ii), we have the following
instability result.

\begin{prop}
\label{unstable}Assume $\left(  A.0\right)  $, $\left(  A.1\right)  $ and
$q\in\mathcal{I}_{\mathcal{N}}$. Then $u_{1,\alpha}$ is the only
weakly stable solution of $(P_{\alpha})$ in $\mathcal{P}^{\circ}$ for
$\alpha\in(0,\alpha_{\ast}]$.
\end{prop}

\noindent\textit{Proof.}
Assume
by contradiction that
for some $\alpha_{0}>0$, $u_{0}\in\mathcal{P}^{\circ}$ is a weakly stable
solution of $(P_{\alpha_{0}})$ and $u_{0}\not \in \mathcal{C}_{1}$. We remark
that $(\alpha_{0},u_{0})$ has a positive distance to the set $\{(\alpha
,u_{1,\alpha}):\alpha\in(0,\alpha_{\ast}]\}$. Then, thanks to Proposition
\ref{prop:bend}, we can assume that $u_{0}$ is asymptotically stable, i.e.,
$\gamma_{1}(\alpha_{0},u_{0})>0$, and $(\alpha_{0},u_{0})\not \in
\{(\alpha,u_{1,\alpha}):\alpha\in(0,\alpha_{\ast}]\}$.

Now, by the IFT, we obtain a solution curve $(\alpha,v(\alpha))$ with
$v(\alpha_{0})=u_{0}$, $v(\alpha)\in\mathcal{P}^{\circ}$ and $\gamma
_{1}(\alpha,v(\alpha))>0$, parametrized by $\alpha\in(\alpha_{0}-\delta
_{0},\alpha_{0}]$ for some $\delta_{0}>0$. Note that
such curve can be extended to $\alpha\in(0,\alpha_{0}]$. Indeed, otherwise
there exists some $\underline{\alpha}\in\left(  0,\alpha_{0}\right)  $ such
that either $v(\underline{\alpha})\not \in \mathcal{P}^{\circ}$ (which
contradicts Proposition \ref{lbuN}(i)) or $\gamma_{1}\left(  \underline
{\alpha},v\left(  \underline{\alpha}\right)  \right)  =0$ (which contradicts
the exactness assertion in Proposition \ref{prop:bend}).

However,
this curve never meets $\{(\alpha,u_{1,\alpha}):\alpha\in(0,\alpha_{\ast}]\}$
(if so we reach a contradiction using the IFT). Consequently, Proposition
\ref{lbuN}(ii) implies that $v(\alpha)=u_{2,\alpha}$ for all $\alpha>0$ small.
But, by Proposition \ref{monoton}(ii), $u_{2,\alpha}$ is unstable if $\alpha$
is small enough, so we reach a contradiction. \qed\newline

\begin{prop}
\label{ast:es}Assume $\left(  A.0\right)  $, $\left(  A.1\right)  $ and
$q\in\mathcal{I}_{\mathcal{N}}$. Then $\alpha_{\ast}=\alpha_{s}$ and
$u_{1,\alpha_{s}}$ is the unique solution of $(P_{\alpha_{s}})$ in
$\mathcal{P}^{\circ}$.
\end{prop}

\noindent\textit{Proof.} First we verify that $u_{1,\alpha_{\ast}}$ is the
unique solution in $\mathcal{P}^{\circ}$ of $(P_{\alpha_{\ast}})$, using the
sub and supersolutions method \cite[Proposition 7.8]{Am76} for a fixed point
equation which $(P_{\alpha})$ is reduced to (see Appendix \ref{appen}). Assume
to the contrary that $v_{0}$ is a solution in $\mathcal{P}^{\circ}$ of
$(P_{\alpha_{\ast}})$ such that $v_{0}\not \equiv u_{\ast}:=u_{1,\alpha_{\ast
}}$. \newline

The argument is divided into two cases.

\begin{enumerate}
\item[(I)] Assume that $v_{0}\geq u_{\ast}$. In this case, using the strong
maximum principle and Hopf's lemma, we deduce $v_{0}-u_{\ast}>0$ in
$\overline{\Omega}$. So, from Proposition \ref{prop:bend} we can take a
solution $u(t)\in\mathcal{P}^{\circ}$ of $(P_{\alpha\left(  t\right)  })$,
satisfying $\alpha(t)<\alpha_{\ast}$, $u_{\ast}<u(t)<v_{0}$ in $\overline
{\Omega}$ if $t>0$ small. We see $u(t)$ is a subsolution of $(P_{\alpha_{\ast
}})$. Using
\cite[Proposition 7.8]{Am76} and taking into account (\ref{a})
in Appendix \ref{appen}, we have a weakly stable solution $u\in\mathcal{P}%
^{\circ}$ of $(P_{\alpha_{\ast}})$ such that $u(t)\leq u\leq v_{0}$ in
$\overline{\Omega}$. However, this is contradictory to Proposition
\ref{unstable}, since $u\not \equiv u_{\ast}$.

\item[(II)] Assume that $v_{0}\not \geq u_{\ast}$. Let $v_{0}\wedge u_{\ast
}:=\min(v_{0},u_{\ast})$. Observe that $v_{0}\not \geq u_{\ast}$ implies that
$v_{0}\wedge u_{\ast}<u_{\ast}$ somewhere, i.e., $v_{0}\wedge u_{\ast}\leq
u_{\ast}$ and $v_{0}\wedge u_{\ast}\not \equiv u_{\ast}$. So, $v_{0}\wedge
u_{\ast}\in\mathcal{P}^{\circ}$ is a
supersolution of $(P_{\alpha_{\ast}})$ in the sense mentioned in Appendix
\ref{appen}.
Moreover, by Proposition \ref{lbuN}(i) we have that $u_{\mathcal{N}}\leq
v_{0}\wedge u_{\ast}$.

Since $u_{\mathcal{N}}$ is a subsolution of $(P_{\alpha_{\ast}})$,
we infer that $(P_{\alpha_{\ast}})$ has a weakly stable solution
$u\in\mathcal{P}^{\circ}$ such that $u_{\mathcal{N}}\leq u\leq v_{0}\wedge
u_{\ast}$ and $u\not \equiv u_{\ast}$. But, this is impossible by Proposition
\ref{unstable}.
\end{enumerate}

The uniqueness of $u_{1,\alpha_{\ast}}$ has been verified. \newline

Now, we can prove in the same way that $(P_{\alpha})$ has no solutions in
$\mathcal{P}^{\circ}$ for any $\alpha>\alpha_{\ast}$. Indeed, if $v_{0}%
\in\mathcal{P}^{\circ}$ is a solution of $(P_{\alpha})$ for some
$\alpha>\alpha_{\ast}$, then $v_{0}\not \equiv u_{\ast}$. Moreover, the above
arguments in (I), (II) remain valid for this $v_{0}$. We have now proved that
$\alpha_{\ast}=\alpha_{s}$. \qed\newline

\begin{rem}
\label{minim}Arguing as in the proof of Proposition \ref{ast:es} we see that,
under the assumptions of the aforementioned proposition, $u_{1,\alpha}$ is
minimal among the solutions in $\mathcal{P}^{\circ}$ of $\left(  P_{\alpha
}\right)  $ for $\alpha\in\left(  0,\alpha_{s}\right]  $. Indeed, if for some
$\alpha\in\left(  0,\alpha_{s}\right]  $ there exists a solution $v_{0}%
\in\mathcal{P}^{\circ}$ of $\left(  P_{\alpha}\right)  $ with $u_{1,\alpha
}\not \leq v_{0}$, then reasoning as in (II) above, with $u_{\ast}$ replaced
by $u_{1,\alpha}$, we reach a contradiction. \smallskip
\end{rem}

Let us
call $\mathcal{C}_{1},\mathcal{C}_{2}$ and $\mathcal{C}_{3}$ the solutions
curves of $\left(  P_{\alpha}\right)  $ given by Propositions \ref{lbuN}(iia),
\ref{lbuN}(iib) and \ref{prop:bend}, respectively.

\begin{prop}
\label{prop:compo:wthoutA2}Assume $(A.0),(A.1)$, $q\in\mathcal{I}%
_{\mathcal{N}}$, and $0\not \equiv a\geq0$ in some smooth domain
$D\subset\Omega$ such that $\left\vert \partial D\cap\partial\Omega\right\vert
>0$. Then $(P_{\alpha})$ has a component $\mathcal{C}_{\ast}$ of solutions in
$[0,\alpha_{s}]\times\mathcal{P}^{\circ}$ which includes $\mathcal{C}%
_{1},\mathcal{C}_{2}$ and $\mathcal{C}_{3}$, satisfying \eqref{compoCast}.
\end{prop}

\noindent\textit{Proof.} Arguing as in the proof of \cite[Theorem
4.4]{KRQU2019}, we have a subcontinuum $\gamma_{0}$ in $[0,\infty)\times
C^{1}(\overline{\Omega})$ of solutions of $(R_{\alpha})$ for $\alpha>0$,
satisfying $\gamma_{0}\cap\{(\alpha
,w)=(0,c):\mbox{$c\geq 0$ is a constant}\}=\{(0,0),(0,c_{a})\}$.

Now we see that $\gamma_{0}^{+}:=\gamma_{0}\setminus\{(0,0),(0,c_{a})\}$ is
connected. By the change of variables $u=\alpha^{-\frac{1}{1-q}}w$, we convert
$\gamma_{0}^{+}$ to a connected set $\mathcal{C}_{+}$ of solutions of
$(P_{\alpha})$, which contains $(\alpha,u_{2,\alpha})$ with $\alpha
\in(0,\overline{\alpha})$ by construction (see \cite[Proposition
3.11]{KRQU2019}, Proposition \ref{lbuN}(iib)).
We assert here that $\mathcal{C}_{+}$ consists of solutions in $\mathcal{P}%
^{\circ}$. Indeed, set
\[
E:=\{(\alpha,u)\in\mathcal{C}_{+}:u\in\mathcal{P}^{\circ}\}.
\]
Note that $E\neq\emptyset$ since $u_{2,\alpha}\in\mathcal{P}^{\circ}$. Then,
we shall show that $E$ is open and closed in $\mathcal{C}_{+}$ with respect to
the norm in $\mathbb{R}\times C^{1}(\overline{\Omega})$. Once this holds, we
deduce that $E=\mathcal{C}_{+}$, as desired.

First, from the definition of $E$, we see that $E$ is open in $\mathcal{C}%
_{+}$. Next we verify that $E$ is closed in $\mathcal{C}_{+}$. Assume that
$(\alpha_{i},u_{i})\in E$ converges to some $(\alpha_{\infty},u_{\infty}%
)\in\mathcal{C}_{+}$. Then, from Proposition \ref{lbuN}(i) we deduce that
$u_{i} \geq u_{\mathcal{N}}$ for all $i$. It follows, by passing to the limit,
that $u_{\infty}\geq u_{\mathcal{N}}$, so $(\alpha_{\infty},u_{\infty})\in E$,
as desired.

Finally, let $\mathcal{C}_{\ast}$ be the component of solutions in
$\mathcal{P}^{\circ}$ of $(P_{\alpha})$ in the space $[0,\infty)\times
C^{1}(\overline{\Omega})$ such that $\mathcal{C}_{\ast}\supset\mathcal{C}_{+}%
$. Then, we can check that $\mathcal{C}_{\ast}$ satisfies \eqref{compoCast}
and includes $\mathcal{C}_{1},\mathcal{C}_{2}$ and $\mathcal{C}_{3}$,
combining the following results: Proposition \ref{lbuN}(ii), Proposition
\ref{prop:bend}, Proposition \ref{ast:es} and \cite[Proposition 3.2]%
{KRQU2019}.
\qed\newline

From
the above results we now prove Theorem \ref{mthm2}, which leads us to Theorem
\ref{mthm}. \newline

\noindent\textit{Proof of Theorem \ref{mthm2}}: Item (i) follows from
Propositions \ref{lbuN}(ii),
\ref{unstable}, \ref{ast:es} and Remark \ref{minim}. Item (ii) is a
consequence of Propositions \ref{lbuN}(ii) and \ref{monoton}, while item (iii)
follows from Propositions \ref{prop:bend} and \ref{ast:es} where the exactness
follows from Proposition \ref{ast:es}. Finally, the last item follows from
Proposition \ref{prop:compo:wthoutA2}. \qed

\section{Proof of Theorem \ref{ball}}

\label{sec:prth1.3}

Let us recall some results on the following \textit{two} eigenvalue problems,
considered
for a generic sign-changing $m\in C(\overline{\Omega})$:
\begin{equation}%
\begin{cases}
-\Delta\varphi=\lambda m(x)\varphi & \mbox{ in }\Omega,\\
\partial_{\nu}\varphi=\alpha\varphi & \mbox{ on }\partial\Omega,
\end{cases}
\label{epro01}%
\end{equation}
where
$\lambda=\lambda(\alpha,m)$ is an eigenvalue; and, for $\lambda\in\mathbb{R}$
fixed,
\begin{equation}%
\begin{cases}
-\Delta\phi=\lambda m(x)\phi+\mu\phi & \mbox{ in }\Omega,\\
\partial_{\nu}\phi=\alpha\phi & \mbox{ on }\partial\Omega,
\end{cases}
\label{epro02}%
\end{equation}
where $\mu=\mu(\lambda,\alpha,m)$ is an eigenvalue. It is well known that
\eqref{epro02} has a sequence of eigenvalues
\[
\mu_{1}<\mu_{2}\leq\mu_{3}\leq...
\]
It is easy to see that, for every $\alpha$,
\[
\mu_{k}(\lambda,\alpha,m)=0\mbox{ for some }k\geq1\ \Longleftrightarrow
\mbox{ $\lambda$ is an eigenvalue of \eqref{epro01}}.
\]
Moreover, the mapping $\lambda\mapsto\mu_{1}(\lambda,\cdot,\cdot)$ is concave,
and the mapping $\lambda\mapsto\mu_{2}(\lambda,\cdot,\cdot)$ is continuous and
satisfies $\mu_{2}(\lambda,\alpha,u)\rightarrow-\infty$ as $\lambda
\rightarrow\pm\infty$
(see e.g. \cite[Lemmas 3.3 and 7.2, Theorem 5.1]{RR}).

Recall that an eigenvalue of \eqref{epro01} or \eqref{epro02} is
\textit{principal} if it possesses a positive eigenfunction (which is in
$\mathcal{P}^{\circ}$ by the strong maximum principle and Hopf's lemma).
It is well known that $\mu_{1}$ is principal and simple, and any other
$\mu_{k}$ ($k\neq1$) is not principal. So, the principal eigenvalues of
\eqref{epro01} appear if and only if $\mu_{1}(\lambda,\alpha,m)=0$.

In \cite[Theorem 5]{AB99} the authors proved that under $(A.0)$ there exists
$\beta_{0}=\beta_{0}(m)>0$ such that \eqref{epro01} has a principal eigenvalue
if and only if $\alpha\leq\beta_{0}$. Additionally if $\alpha<\beta_{0}$ then
the principal eigenvalues of this problem are given exactly by
\begin{align}
\lambda_{\pm1}  &  =\lambda_{\pm1}(m,\alpha)\nonumber\\
&  =\pm\inf\left\{  \frac{\int_{\Omega}|\nabla\phi|^{2}-\alpha\int
_{\partial\Omega}\phi^{2}}{\int_{\Omega} m(x)\phi^{2}}:\phi\in H^{1}%
(\Omega),\ \int_{\Omega}m(x)\phi^{2}\gtrless0\right\}  . \label{def:lam+-1}%
\end{align}
These infima are achieved by $\phi_{\pm1}=\phi_{\pm1}(m,\alpha)$, normalized
as $\int_{\Omega}(\phi_{\pm1})^{2}=1$, which in particular satisfy
\begin{equation}
\int_{\Omega}m(x)(\phi_{\pm1})^{2}\gtrless0. \label{nonzero}%
\end{equation}
In addition,
$0<\lambda_{-1}(m,\alpha)<\lambda_{+1}(m,\alpha)$ for $0<\alpha<\beta_{0}$.

Let $\mu_{k}(\alpha):=\mu_{k}(0,\alpha,m)$ (note that $\mu_{k}$ does not
depend on $m$ when $\lambda=0$). It is well known that $\mu_{1}(\alpha
),\mu_{2}(\alpha)$ are characterized by the following variational formulas,
respectively:
\[
\mu_{1}(\alpha)=\min_{\overset{\phi\in H^{1}(\Omega)}{\Vert\phi\Vert_{2}=1}%
}\left(  \int_{\Omega}|\nabla\phi|^{2}-\alpha\int_{\partial\Omega}\phi
^{2}\right)  ,
\]%
\begin{equation}
\mu_{2}(\alpha)=\min_{\overset{E\subset H^{1}(\Omega)}{dim(E)=2}}%
\max_{\overset{\phi\in E}{\Vert\phi\Vert_{2}=1}}\left(  \int_{\Omega}%
|\nabla\phi|^{2}-\alpha\int_{\partial\Omega}\phi^{2}\right)  . \label{mu2c}%
\end{equation}

Let us denote by
$\{\alpha_{j}\}$ the sequence of eigenvalues of the Steklov problem
\eqref{st}. We note that $\alpha_{j}$ is an eigenvalue of \eqref{st} if and
only if $\mu_{k}(\alpha_{j})=0$ for some $k\geq1$. Note also that $\alpha
_{1}=0$ is the only principal eigenvalue of \eqref{st}.

We prove some useful results on $\mu_{2}(\alpha)$ and $\lambda_{2}^{-}%
(\alpha,m)$:

\begin{prop}
\label{prop:mu2} The following three assertions hold:

\begin{enumerate}
\item $\alpha\mapsto\mu_{2}(\alpha)$ is
non-increasing for $\alpha\geq0$, i.e., $\mu_{2}(\alpha)\geq\mu_{2}(\beta)$
for $0\leq\alpha<\beta$.

\item We have
\[
\left\{
\begin{array}
[c]{ll}%
\mu_{2}(\alpha)>0, & \text{for }\alpha\in(0,\alpha_{2}),\\
\mu_{2}(\alpha)=0, & \text{for }\alpha=\alpha_{2},\\
\mu_{2}(\alpha)<0, & \text{for }\alpha\in(\alpha_{2},\infty).
\end{array}
\right.
\]

\item $\lambda_{2}^{-}(\alpha,m)<0$ for $0<\alpha<\alpha_{2}$ and any
sign-changing $m\in C(\overline{\Omega})$.
\end{enumerate}
\end{prop}



\textit{Proof.}

\begin{enumerate}
\item By \eqref{mu2c} it is clear that $\mu_{2}(\alpha)\geq\mu_{2}(\beta)$ for
$0\leq\alpha<\beta$.

\item First of all, it is clear that $\mu_{2}(\alpha_{2})=0$. Assume now that
$\mu_{2}(\alpha)\leq0 $ for some $0<\alpha<\alpha_{2}$. Since $\mu_{2}(0)>0$
and $\alpha\mapsto\mu_{2}(\alpha)$ is continuous, we infer that there exists
$0<\beta\leq\alpha$ such that $\mu_{2}(\beta)=0$. Thus $\beta$ is a
non-principal eigenvalue of \eqref{st}, which contradicts $\beta<\alpha_{2}$.
Finally, let $\alpha>\alpha_{2}$. By the previous item we have $\mu_{2}%
(\alpha)\leq0$ for $\alpha>\alpha_{2}$. However, if $\mu_{2}(\alpha)=0$ then
$\mu_{2}$ vanishes in $(\alpha_{2},\alpha)$, which contradicts the
discreteness of the spectrum of \eqref{st}.

\item By the previous item,
it suffices to prove that $\lambda_{2}^{-}(\alpha,m)<0$ (which is equivalent
to $\lambda_{2}^{+}(\alpha,-m)>0$) if $\mu_{2}(\alpha)>0$. To do so, we shall
use a formulation of second eigenvalues that goes back at least to \cite{V}
(see also \cite{T}):
\[
\mu_{2}(\alpha)=\min_{(A,B)\in\mathcal{J}}\max(\mu^{+}(A),\mu^{+}(B)),
\]
where, for any open set $A\subset\Omega$,
\[
\mu^{+}(A):=\inf\left\{  \int_{\Omega}|\nabla u|^{2}-\alpha\int_{\partial
\Omega}u^{2}:u\in H_{A}^{1}(\Omega),\Vert u\Vert_{2}=1\right\}  ,
\]%
\[
H_{A}^{1}(\Omega):=\{u\in H^{1}(\Omega),u=0\mbox{ a.e. in }\Omega
\setminus\overline{A}\},
\]
and
\[
\mathcal{J}:=\{(A,B):A,B\mbox{ are disjoint nonempty open subsets of }\Omega
\}.
\]
In a similar way,
\begin{equation}
\lambda_{2}^{+}(\alpha,-m)=\min_{(A,B)\in\mathcal{J}}\max(\lambda
^{+}(A),\lambda^{+}(B)) \label{e2}%
\end{equation}
where
\[
\lambda^{+}(A):=\inf\left\{  \int_{\Omega}|\nabla u|^{2}-\alpha\int
_{\partial\Omega}u^{2}:u\in H_{A}^{1}(\Omega),\int_{\Omega}m(x)u^{2}%
=-1\right\}  .
\]
We set $\lambda^{+}(A)=\infty$ if there is no $u$ satisfying the above
constraints. Note that $\lambda^{+}(A)$ is achieved whenever it is finite.
From the formula above we see that $\lambda_{2}^{+}(\alpha,-m)>0$ if, and only
if, given $(A,B)\in\mathcal{J}$, we have either $\lambda^{+}(A)>0$ or
$\lambda^{+}(B)>0$. Now, if $\mu_{2}(\alpha)>0$ then, for such a pair $(A,B)$,
we have either $\mu^{+}(A)>0$ or $\mu^{+}(B)>0$. If $\mu^{+}(A)>0$ then
$\int_{\Omega}|\nabla u|^{2}-\alpha\int_{\partial\Omega}u^{2}>0$ for every
nontrivial $u\in H_{A}^{1}(\Omega)$. In particular, this inequality holds if,
in addition, $\int_{\Omega}m(x)u^{2}=-1$, which shows that $\lambda^{+}(A)>0$.
Similarly, we see that $\mu^{+}(B)>0$ implies $\lambda^{+}(B)>0$. Therefore
the minimum in \eqref{e2} is positive, which yields the conclusion.
\qed\newline
\end{enumerate}

\noindent\textit{Proof of Theorem \ref{ball}:}\newline

Let us prove (i). We proceed in several steps. The first one is to show that
the IFT can be applied at $\left(  \alpha,u\right)  $, where $\alpha\in\left(
0,\alpha_{s}\right)  $ and $u\in\mathcal{P}^{\circ}$ is a solution of $\left(
P_{\alpha}\right)  $. Since the IFT can be applied at $u_{1,\alpha}$ and this
is the only stable solution for $\alpha\in\left(  0,\alpha_{s}\right)  $, cf.
Propositions \ref{unstable} and \ref{ast:es}, we may assume without loss of
generality that $\gamma_{1}\left(  \alpha,u\right)  <0$. We consider
\eqref{epro01} with $m=au^{q-1}$. Note that $\lambda=1$ is a principal
eigenvalue of this problem (associated with $u\in\mathcal{P}^{\circ}$).
Moreover, since $\int_{\Omega}au^{q-1}<0$ we have $0<\lambda_{1}^{-}\leq
1\leq\lambda_{1}^{+}$.

By \cite[Proposition 3.4]{KRQU2019} (which clearly holds if $q\in
\mathcal{I}_{\mathcal{N}}$) and (\ref{hip}) we have
\[
\alpha<\alpha_{s}\leq\frac{-\int_{\Omega}a}{\int_{\partial\Omega
}u_{\mathcal{N}}^{1-q}}\leq\alpha_{2}.
\]
So, from Proposition \ref{prop:mu2}(iii), we deduce that $\lambda_{2}%
^{-}=\lambda_{2}^{-}\left(  \alpha,m\right)  <0$. Summing up, the eigenvalue
sequence of \eqref{epro01} satisfies
\[
\cdots\leq\lambda_{3}^{-}\leq\lambda_{2}^{-}<0<\lambda_{1}^{-}\leq1\leq
\lambda_{1}^{+}<\lambda_{2}^{+}\leq\lambda_{3}^{+}\leq\cdots.
\]

We now observe that the IFT can be applied at $\left(  \alpha,u\right)  $.
Indeed, assume by contradiction that $0=\gamma_{k}\left(  \alpha,u\right)  $,
for some $k\geq2$. Then $q\in(0,1)$ is an eigenvalue of \eqref{epro01}.
Moreover, since $k\not =1$, $\phi$ changes sign in $\Omega$, and so
$q\not \in \{\lambda_{1}^{-},\lambda_{1}^{+}\}$. So, since $\lambda_{2}^{-}%
<0$, we have $q=\lambda_{j}^{+}$ for some $j\geq2$, and then $1\leq\lambda
_{1}^{+}<\lambda_{2}^{+}\leq q$,
which is not possible. Therefore, for $\alpha\in\left(  0,\alpha_{s}\right)  $
and a solution $u\in\mathcal{P}^{\circ}$ of $\left(  P_{\alpha}\right)  $, we
see that the IFT can be applied at $\left(  \alpha,u\right)  $, as claimed.

The next step is to prove that the curve $\mathcal{C}_{2}$ can be extended to
all $\alpha\in\left(  0,\alpha_{s}\right)  $. Indeed, take the maximal
$\alpha$ of this curve, say $\alpha_{max}$, and suppose $\alpha_{max}%
<\alpha_{s}$. As $\alpha\nearrow\alpha_{max}$, the family of solutions
$\{u_{\alpha}\}$ is bounded \cite[Proposition 3.2]{KRQU2019} and away from $0$
in $C(\overline{\Omega})$ by
Proposition \ref{lbuN}(i). So, going to the limit we find a nontrivial
solution $u_{\alpha_{max}}$ of $\left(  P_{\alpha_{max}}\right)  $. Moreover,
$u_{\alpha_{max}}\in\mathcal{P}^{\circ}$, and recalling Propositions
\ref{unstable} and \ref{ast:es}
we get that $\gamma_{1}(\alpha_{max},u_{\alpha_{max}})<0$.
We thus apply the IFT at $(\alpha_{max},u_{\alpha_{max}})$ and we get a
contradiction with the maximality of $\alpha_{max}$. Thus, $\mathcal{C}_{2}$
can be extended to $\alpha\in\left(  0,\alpha_{s}\right)  $, as asserted.

The final step is to show that any solution $v$ of $\left(  P_{\alpha}\right)
$, $\alpha\in\left(  0,\alpha_{s}\right)  $, must be either $u_{1,\alpha}$ or
$u_{2,\alpha}$. In order to see this, we argue as above (also similarly as in
the proof of Proposition \ref{unstable}). Assume that $(\alpha,v)\not \in
\mathcal{C}_{1}$. Then, we deduce that $\gamma_{1}(\alpha,v)<0$. Applying the
IFT at $\left(  \alpha,v\right)  $ we obtain a curve of $\mathcal{C}%
_{3}:v=v_{\alpha}\in\mathcal{P}^{\circ}$ of solutions of $\left(  P_{\alpha
}\right)  $. Take the minimal $\alpha$ of this curve, say $\alpha_{min}$, and
suppose $\alpha_{min}>0$. Reasoning as before we obtain some solution
$v_{\alpha_{min}}\in\mathcal{P}^{\circ}$ of $\left(  P_{\alpha_{min}}\right)
$. We must have $\gamma_{1}\left(  \alpha_{min},v_{\alpha_{min}}\right)  <0$,
but, in this case, applying again the IFT we get a contradiction. So,
$\alpha_{min}=0$. Now, by the exactness results for $\alpha>0$ small and the
condition $\gamma_{1}(\alpha,v_{\alpha})<0$, we have $v_{\alpha}=u_{2,\alpha}$
for such $\alpha$, implying $(\alpha,v)=(\alpha,u_{2,\alpha})\in
\mathcal{C}_{2}$.

To conclude the proof we note that, if $\partial\Omega\subseteq\partial
\Omega_{+}^{a}$ and $q\in\mathcal{A}_{\mathcal{N}}$, then from \cite[Lemma
2.1(i) and Proposition 2.3]{KRQU2019} we have that $q\in\mathcal{A}_{\alpha}$
for all $\alpha>0$. In other words, any nontrivial solution of $(P_{\alpha})$
is in $\mathcal{P}^{\circ}$ for every $\alpha>0$, and (ii) follows. \qed


\section{Bifurcaton analysis as $q \to1^{-}$}

\label{sec:bif}

This section is devoted to further analysis of the solutions set in
$\mathcal{P}^{\circ}$ of $(P_{\alpha})$ as $q\rightarrow1^{-}$ for $\alpha
\in(0,\beta_{0})$ fixed, by recalling the critical value $\beta_{0}>0$ from
Section \ref{sec:prth1.3}. Moreover, we show how to deduce the asymptotic
behavior of $\alpha_{s}(a,q)$
(given by \eqref{def:alphs}) as $q\rightarrow1^{-}$.

Throughout
this section it will be convenient to rename $\left(  P_{\alpha}\right)  $ as
$\left(  P_{a,q,\alpha}\right)  $ or simply $\left(  P_{q}\right)  $ if no
confusion arises. We shall handle the eigenvalue problem \eqref{epro01} with
$m=a$. Assume $(A.0)$, $q\in(0,1)$ and $\alpha\in\left(  0,\beta_{0}\right)
$.
We look at $q$ as a bifurcation parameter in $\left(  P_{q}\right)  $,
assuming that
\begin{equation}
\lambda_{+1}(a,\alpha)=1 \label{def:lam1}%
\end{equation}
to obtain solutions of $(P_{q})$ bifurcating from the line $(q,u)=(1,t\phi
_{+1})$ where $\lambda_{+1}$ and $\phi_{+1}$ are from \eqref{def:lam+-1} and
\eqref{nonzero}, respectively. A similar procedure can be applied when
$\lambda_{-1}(a,\alpha)=1$.

Recall that
$\lambda_{+1}$ is simple, so that setting
\[
A:=-\Delta-\lambda_{+1}a(x),\quad C_{\alpha}^{2+r}(\overline{\Omega
}):=\left\{  u\in C^{2+r}(\overline{\Omega}):B_{\alpha}u:=\frac{\partial
u}{\partial\nu}-\alpha u=0\mbox{ on }\partial\Omega\right\}  ,
\]
we have that $\mathrm{Ker}A=\langle\phi_{+1}\rangle:=\{t\, \phi_{+1}%
:t\in\mathbb{R}\}$. Condition \eqref{def:lam1} implies that $(P_{q})$
possesses the trivial line $\Gamma_{1}:=\{(1,t\, \phi_{+1}),t>0\}$. As in the
Neumann case $\alpha=0$ (see \cite[Section 2]{KRQUnodea}), we employ the
Lyapunov-Schmidt type reduction to analyze the bifurcating solutions in
$\mathcal{P}^{\circ}$ from $\Gamma_{1}$. The usual decomposition of
$C_{\alpha}^{2+r}(\overline{\Omega})$ is given by the formula $C_{\alpha
}^{2+r}(\overline{\Omega})=\langle\phi_{+1}\rangle+X_{2};u=t\,\phi_{+1}+w,$
where $t=\int_{\Omega}u\phi_{+1}$, and $w=u-(\int_{\Omega}u\phi_{+1})\phi
_{+1}$. So, $X_{2}$ is characterized as $X_{2}=\left\{  w\in C_{\alpha}%
^{2+r}(\overline{\Omega}):\int_{\Omega}w\phi_{+1}=0\right\}  .$ On the other
hand, put
$C^{r}(\overline{\Omega})=Y_{1}+\mathrm{Im}(A)$, where $Y_{1}=\langle\phi
_{+1}\rangle$, and
$\mathrm{Im}(A):=\left\{  f\in C^{r}(\overline{\Omega}):\int_{\Omega}%
f\phi_{+1}=0 \right\}  $. Let $Q$ be the projection of $Y$ to $\mathrm{Im}%
(A)$, given by $Q[f]:=f-\left(  \int_{\Omega}f\phi_{+1}\right)  \phi_{+1}$.
Using $Q$, we
reduce $(P_{q})$ to the following coupled equations:
\begin{align*}
&  Q[Au]=Q[a\left(  x\right)  \left(  u^{q}-u\right)  ],\\
&  (1-Q)[Au]=(1-Q)[a\left(  x\right)  \left(  u^{q}-u\right)  ].
\end{align*}
For $u=t\,\phi_{+1}+w$ the first equation yields
\begin{equation}
-\Delta w-a(x)w=Q[a\left(  x\right)  \{(t\,\phi_{+1}+w)^{q}-(t\,\phi
_{+1}+w)\}]. \label{beq01}%
\end{equation}
The second equation yields that
\begin{equation}
\int_{\Omega}a(x)\left\{  (t\,\phi_{+1}+w)^{q}-(t\,\phi_{+1}+w)\right\}
\phi_{+1}=0. \label{beq02}%
\end{equation}

Now, we see that $(q,t,w)=(1,t_{0},0)$ solves \eqref{beq01} and \eqref{beq02}
for any $t_{0}>0$. The IFT is applicable to \eqref{beq01} at $(1,t_{0},0))$
(as in \cite[Section 2]{KRQUnodea}), and then, \eqref{beq01} is solved exactly
by $w=w(q,t)$ around this point such that $w(1,t_{0})=0$. We plug $w(q,t)$
into \eqref{beq02} to get the following bifurcation equation in $\mathbb{R}%
^{2}$:
\[
\Phi(q,t):=\int_{\Omega}a(x)\{(t\,\phi_{+1}+w(q,t))^{q}-(t\,\phi
_{+1}+w(q,t)\}\phi_{+1}=0,\quad(q,t)\simeq(1,t_{0}).
\]

Set now
\begin{equation}
t_{\pm}=t_{\pm}(a,\alpha):=\exp\left[  -\frac{\int_{\Omega}a\left(  x\right)
(\phi_{\pm1})^{2}\log\phi_{\pm1}}{\int_{\Omega}a\left(  x\right)  (\phi_{\pm
1})^{2}}\right]  , \label{et*}%
\end{equation}
and we are in position to state the next result.

\begin{prop}
\label{t1} Assume $(A.0)$, $\alpha\in(0,\beta_{0})$
and \eqref{def:lam1}. Then the following assertions hold:

\begin{enumerate}
\item Assume that $(q_{n},u_{n})\in(0,1)\times\mathcal{P}^{\circ}$ are
solutions of $(P_{q_{n}})$ such that $(q_{n},u_{n})\rightarrow(1,t\,\phi_{+1}
)\in\Gamma_{1}$ in $\mathbb{R}\times C^{2+r}(\overline{\Omega})$ for some
$t>0$. Then, we have $t=t_{+}$, where $t_{+}$ is given by \eqref{et*}.

\item The set of solutions in $\mathcal{P}^{\circ}$ of $(P_{q})$ consists of
$\Gamma_{1}\cup\Gamma_{2}$ in a neighborhood of $(q,u)=(1,t_{+}\,\phi_{+1})$
in $\mathbb{R}\times C^{2+r}(\overline{\Omega})$, where
\[
\Gamma_{2}:=\{(q,\,t(q)\phi_{+1}+w(q,t(q))):|q-1|<\delta_{\ast}\}\quad
\mbox{for some $\delta_{\ast} > 0$.}
\]
Here $t(q)$ and $w(q,t(q))$ are smooth with respect to $q$ and satisfy
$t(1)=t_{+}$ and $w(1,t_{+})=0$.
\end{enumerate}
\end{prop}

Proposition
\ref{t1} can be proved in the same way as \cite[Theorem 2.2]{KRQUnodea}. We
only have to note that condition \eqref{nonzero}
for $\phi_{+1}$ is used
essentially in the proof of (ii) for applying the IFT to $\hat{\Phi
}(q,t)=\frac{\Phi(q,t)}{q-1}$ at $(1,t_{+})$. \newline

As a consequence of the previous
result, we obtain the following:

\begin{theorem}
\label{tt2}

Assume $(A.0)$.
Then:

\begin{enumerate}
\item For any $\alpha\in(0, \beta_{0})$ the problem $(P_{\alpha})$ has at
least two solutions $U_{1,q}=U_{1,q}(a,\alpha), U_{2,q}=U_{2,q}(a,\alpha)$ in
$\mathcal{P}^{\circ}$ for $q$ close to $1$. These solutions satisfy
$U_{2,q}-U_{1,q} \in\mathcal{P}^{\circ}$, and
\[
U_{1,q} \sim\lambda_{+1}^{-\frac{1}{1-q}} \, t_{+} \phi_{+1}, \quad U_{2,q}
\sim\lambda_{-1}^{-\frac{1}{1-q}} \, t_{-} \phi_{-1} \quad\mbox{ for } q
\sim1,
\]
i.e.
\[
\lambda_{+1}(a,\alpha)^{\frac{1}{1-q}} U_{1,q} \rightarrow t_{+} \phi_{+1}
\quad\mbox{and}\quad\lambda_{-1}(a,\alpha)^{\frac{1}{1-q}} U_{2,q} \rightarrow
t_{-} \phi_{-1}
\]
in $C^{2+r}(\overline{\Omega})$ as $q \to1^{-}$ for some $r\in(0,1)$. More precisely:

\begin{enumerate}
\item If $\lambda_{+1}=1$, then $U_{1,q} \rightarrow t_{+} \phi_{+1}$ in
$C^{2+r}(\overline{\Omega})$ as $q\rightarrow1^{-}$;

\item If $\lambda_{+1}>1$, then $U_{1,q} \rightarrow0$ in $C^{2+r}%
(\overline{\Omega})$ as $q\rightarrow1^{-}$;

\item If $\lambda_{+1}<1$, then $\displaystyle\min_{\overline{\Omega}}%
U_{1,q}\rightarrow\infty$ as $q\rightarrow1^{-}$;
\end{enumerate}

and a similar result holds for $U_{2,q}$.

\item Assume in addition $\left(  A.1\right)  $. Then $\alpha_{s}
(a,q)\rightarrow\beta_{0}$ as $q\rightarrow1^{-}$.\medskip
\end{enumerate}
\end{theorem}

\begin{proof} \strut
\begin{enumerate}
\item Set $v:=\lambda_{+1}(a,\alpha)^{\frac{1}{1-q}}u$. Note that if $u$ solves $(P_{q})$ then $v$ solves $(P_{q})$ with $a$ replaced by $\tilde{a}:=\lambda_{+1}(a,\alpha)a$. Indeed,
\[
-\Delta v=\lambda_{+1}(a,\alpha)^{\frac{1}{1-q}}a\left(  x\right)
u^{q}=\lambda_{+1}(a,\alpha) \, a(x)v^{q}
=\tilde{a}\left(  x\right) v^{q}.
\]
Moreover, we easily see that $\lambda_{+1}(\tilde{a},\alpha)=1$.
By 
Proposition \ref{t1}, we get a solution $v_{1,q}\in\mathcal{P}^{\circ}$ of $(P_{q})$ with
$\tilde{a}$ such that $v_{1,q}\rightarrow t_{+}(\tilde{a})\phi_{+1}(\tilde{a})$
as $q\rightarrow1^{-}$. But it is easily seen that $\phi_{+1}(\tilde{a}%
)=\phi_{+1}(a)$ and $t_{+}(\tilde{a})=t_{+}(a)$.
Thus we obtain a solution $U_{1,q}=\lambda_{+1}(a,\alpha)^{\frac{1}{q-1}%
}v_{1,q}\in\mathcal{P}^{\circ}$ of $(P_{q})$ for $q$ close to $1$ such that
\[
\lambda_{+1}(a,\alpha)^{\frac{1}{1-q}}U_{1,q}\longrightarrow t_{+}(a)\phi_{+1}(a)\quad\mbox{ as }\ q\rightarrow1^{-}.
\]
In particular, we see that if $\lambda_{+1}(a,\alpha)>1$ then $\lambda_{+1}(a,\alpha)^{\frac{1}{q-1}}\rightarrow0$, so that $U_{1,q}\rightarrow0$ in
$C^{2+r}(\overline{\Omega})$ as $q\rightarrow1^{-}$. On the other hand, if
$\lambda_{+1}(a,\alpha)<1$, then $\lambda_{+1}(a,\alpha)^{\frac{1}{q-1}%
}\rightarrow\infty$, so that $\displaystyle\min_{\overline{\Omega}}%
U_{1,q}\rightarrow\infty$ when $q\rightarrow1^{-}$. A similar argument with
$\lambda_{-1}$ instead of $\lambda_{+1}$ provides a solution $U_{2,q}
=\lambda_{-1}(a,\alpha)^{\frac{1}{q-1}}v_{2,q}\in\mathcal{P}^{\circ}$ with
$v_{2,q}\rightarrow t_{-}(a)\phi_{-1}(a)$.
\item Take 
a ball $B\Subset\Omega_{+}^{a}$, and choose $c>0$ large enough so
that $\sigma_{1}^{\mathcal{D}}\left(  ca\right)  <1$, where $\sigma_{1}^{\mathcal{D}}\left(
a\right)  $ denotes the unique positive principal eigenvalue with respect to
the weight $a$ in $B$, under homogeneous Dirichlet boundary conditions.
Note that $\beta_{0}(ca)=\beta_{0}(a)$ and, by a rescaling argument,
$\alpha_{s}(q,ca)=\alpha_{s}(q,a)$ for any $c>0$. So, we may assume that
$\sigma_{1}^{\mathcal{D}}\left(  a\right)  <1$.
%
Since $q\in\mathcal{A}_{\mathcal{N}}$ for $q$ close to $1$, we note from
\cite[Proposition 3.4]{KRQU2019} and \eqref{eun} that
$\alpha_{s}(q)<\frac{-2\sigma_{1}^{\mathcal{N}}(a)\int_{\Omega}a}{|\partial\Omega|}$ as
$q\rightarrow1^{-}$, so $\varlimsup_{q\rightarrow1^{-}}\alpha_{s}(q)<\infty$.
First we prove that
\[
\beta_{0}\leq\varliminf_{q\rightarrow1^{-}}\alpha_{s}(q).
\]
Assume by contradiction that $\varliminf_{q}\alpha_{s}(q)<\beta_{0}$. Then
there exist $q_{n}\nearrow1$ and $\varepsilon>0$ such that $\alpha_{s}%
(q_{n})<\beta_{0}-\varepsilon$. By item (i), we can choose $q_{0}\in(0,1)$
such that $(P_{a,q,\beta_{0}-\varepsilon})$ has a solution in $\mathcal{P}%
^{\circ}$ for every $q\in(q_{0},1)$. In particular, this remains valid for
$q_{n}$, which contradicts the definition of $\alpha_{s}(q_{n})$.
Next we prove that
\begin{equation}
\varlimsup_{q\rightarrow1^{-}}\alpha_{s}(q)\leq\beta_{0}. \label{limsupals}%
\end{equation}
Assume by contradiction that $\beta_{0}<\varlimsup_{q}\alpha_{s}(q)$. Then,
there exist $q_{n}\nearrow1$ and $\varepsilon>0$ such that $\beta
_{0}+\varepsilon\leq\alpha_{s}(q_{n})$. Put $\alpha_{n}:=\alpha_{s}(q_{n})$,
and fix $n$. Let us verify that $(P_{a,q_{n},\alpha_{n}})$ has a solution
$u_{n}\in\mathcal{P}^{\circ}$. By the definition of $\alpha_{n}$, there exist
$\alpha_{i,n}\nearrow\alpha_{n}$ and $u_{i,n}\in\mathcal{P}^{\circ}$ such
that
\[%
\begin{cases}
-\Delta u_{i,n}=a(x)u_{i,n}^{q_{n}} & \mbox{ in }\Omega,\\
\partial_{\nu}u_{i,n}=\alpha_{i,n}u_{i,n} & \mbox{ on }\partial\Omega.
\end{cases}
\]
In view of the fact that (for all $n$ large) $q_{n}\in\mathcal{A}_{\mathcal{N}}%
\subseteq\mathcal{I}_{\mathcal{N}}$, Proposition \ref{lbuN}(i) is applicable to
$u_{i,n}$, and then, $u_{i,n}\geq u_{\mathcal{N}}$. So, using \cite[Proposition
3.2]{KRQU2019}, we may deduce the existence of $u_{n}\in\mathcal{P}^{\circ}$
such that $u_{i,n}\rightarrow u_{n}$ in $C^{1}(\overline{\Omega})$ and  
$u_{n}\geq u_{\mathcal{N}}$, and thus, $u_{n}$ is a desired solution.
We assert that $\Vert u_{n}\Vert:=\Vert u_{n}\Vert_{H^{1}(\Omega)}$ is
bounded. Assume to the contrary that $\Vert u_{n}\Vert\rightarrow\infty$ and
set $v_{n}:=\frac{u_{n}}{\Vert u_{n}\Vert}$. Then $\Vert v_{n}\Vert=1$ and, up
to a subsequence, we get some $\hat{v}\in H^{1}(\Omega)$ such that
$v_{n}\rightharpoonup\hat{v}$ in $H^{1}(\Omega)$, $v_{n}\rightarrow\hat{v}$ in
$L^{2}(\Omega)$ and in $L^{2}(\partial\Omega)$, and $v_{n}\rightarrow\hat{v}$
a.e. From
\begin{equation}
\int_{\Omega}\nabla v_{n}\nabla w=\left(  \int_{\Omega}av_{n}^{q_{n}}w\right)
\Vert u_{n}\Vert^{q_{n}-1}+\alpha_{n}\int_{\partial\Omega}v_{n}w,\quad\forall
w\in H^{1}(\Omega), \label{nw}%
\end{equation}
it follows by passing to the limit that
\begin{equation}
\int_{\Omega}\nabla\hat{v}\nabla w=\hat{b}\int_{\Omega}a\hat{v}w+\hat{\alpha
}\int_{\partial\Omega}\hat{v}w, \label{hatw}%
\end{equation}
where $\hat{b}:=\lim\Vert u_{n}\Vert^{q_{n}-1}\in\lbrack0,1]$ and $\hat
{\alpha}:=\lim\alpha_{n}\in\lbrack\beta_{0}+\varepsilon,\infty)$ (taking a
suitable subsequence, since $\alpha_{n}$ has an upper bound). In addition, we
note that $\int_{\Omega}av_{n}^{q_{n}}w\rightarrow\int_{\Omega}a\hat{v}w$ by
Lebesgue's dominated convergence theorem.
Assertion \eqref{hatw} means that, in weak sense,
\begin{equation}%
\begin{cases}
-\Delta\hat{v}=\hat{b}a(x)\hat{v} & \mbox{ in }\Omega,\\
\partial_{\nu}\hat{v}=\hat{\alpha}\hat{v} & \mbox{ on }\partial\Omega.
\end{cases}
\label{hatv}%
\end{equation}
Moreover, $\hat{v}\not \equiv 0$.
Indeed, using \eqref{nw} and \eqref{hatw}, we deduce that $\int_{\Omega
}|\nabla(v_{n}-\hat{v})|^{2}\rightarrow0$, so $v_{n}\rightarrow\hat{v}$ in
$H^{1}(\Omega)$. Consequently, $\Vert\hat{v}\Vert=1$, as desired. Hence,
$\hat{b}$ is a principal eigenvalue of \eqref{hatv}, but this problem has no
principal eigenvalue since $\beta_{0}+\varepsilon\leq\hat{\alpha}$.
Therefore $\Vert u_{n}\Vert$ is bounded and we may assume that $u_{n}%
\rightharpoonup\hat{u}$ in $H^{1}(\Omega)$, $u_{n}\rightarrow\hat{u}$ in
$L^{2}(\Omega)$ and in $L^{2}(\partial\Omega)$, and $u_{n}\rightarrow\hat{u}$
a.e. From
\[
\int_{\Omega}\nabla u_{n}\nabla w=\int_{\Omega}au_{n}^{q_{n}}w+\alpha_{n}%
\int_{\partial\Omega}u_{n}w,\quad\forall w\in H^{1}(\Omega),
\]
it follows that
\[
\int_{\Omega}\nabla\hat{u}\nabla w=\int_{\Omega}a\hat{u}w+\hat{\alpha}%
\int_{\partial\Omega}\hat{u}w,
\]
meaning that $\hat{u}$ is a weak solution of%
\begin{equation}%
\begin{cases}
-\Delta\hat{u}=a\hat{u} & \mbox{ in }\Omega,\\
\partial_{\nu}\hat{u}=\hat{\alpha}\hat{u} & \mbox{ on }\partial\Omega.
\end{cases}
\label{hatu}%
\end{equation}
Moreover, we can deduce $\hat{u}\not \equiv 0$. Indeed, since $u_{n}>0$ in $B$
and $\sigma_{1}^{\mathcal{D}}(a)<1$, \cite[Lemma 2.5]{KRQU16} provides us with a ball
$B^{\prime}\Subset B$ and $M>0$ such that $u_{n}>M$ in $B^{\prime}$ for all
$n$. The condition $u_{n}\rightarrow\hat{u}$ a.e.\ gives the desired assertion
as $n\rightarrow\infty$.
Consequently, $1$ is a principal eigenvalue of \eqref{hatu}, and we reach once
again a contradiction with $\beta_{0}+\varepsilon\leq\hat{\alpha}$. We have
proved \eqref{limsupals}, which concludes the proof.
\end{enumerate}
\end{proof}

\begin{cor}
\label{cor:asympt} Whenever $u_{1,\alpha}$ and $u_{2,\alpha}$ are the only
solutions in $\mathcal{P}^{\circ}$ for $\alpha\in(0,\beta_{0})$ and $q$ close
to $1$, we have $u_{1,\alpha}=U_{1,q}$ and $u_{2,\alpha}=U_{2,q}$, i.e.
\[
u_{1,\alpha}(q)\sim\lambda_{+1}^{-\frac{1}{1-q}}\,t_{+}\,\phi_{+1},\quad
u_{2,\alpha}(q)\sim\lambda_{-1}^{-\frac{1}{1-q}}\,t_{-}\,\phi_{-1}%
\quad\mbox{ for }\ q\sim1.
\]
In particular, this holds for the weights $a$ built in Section \ref{sec:exa}.
\end{cor}

\begin{rem}
\strut

\begin{enumerate}
\item In view of Remark \ref{r1}(i), we infer from Theorem \ref{tt2}(ii) that
Theorem \ref{ball} holds also if, instead of \eqref{hip}, we assume that
$\beta_{0}(a)<\alpha_{2}$ and $q$ is close enough to $1$.

\item The following characterization can be established:
\[
\beta_{0}(a)=\inf\left\{  \int_{\Omega}|\nabla\phi|^{2}:\phi\in H^{1}%
(\Omega),\ \int_{\Omega}a(x)\phi^{2}=0,\int_{\partial\Omega}\phi
^{2}=1\right\}  .
\]
One can easily show that this infimum is achieved. In addition, if $\{a_{k}\}$
is a sequence of sign-changing weights satisfying $(A.1)$ and $a_{k}%
\rightarrow a$ in $C^{\theta}(\overline{\Omega})$, where $a$ is a negative
weight such that $a<0$ on $\partial\Omega$, then $\beta_{0}(a_{k}%
)\rightarrow\infty$. This result, combined with Theorem \ref{tt2}(ii), shows
that $\alpha_{s}(a_{k},q)$ can be made arbitrarily large by letting
$k\rightarrow\infty$ and $q\rightarrow1^{-}$.
\end{enumerate}
\end{rem}

\section{The solution set structure of $(S_{\alpha})$}

\label{sec:S}

Next we provide a description of the nontrivial solution set of $(S_{\alpha})$
with $\alpha\geq0$, which can be established proceeding in a similar way as
for $(P_{\alpha})$. We note that $(S_{\alpha})$ was mainly investigated by
Alama \cite{alama}, under some conditions on $a$ very similar to $(A.0)$ and
$(A.1)$, which are assumed in this Section.

The notion of stability for solutions of $(S_{\alpha})$ in $\mathcal{P}%
^{\circ}$ can be easily adapted from the one introduced in Section
\ref{sec:pr1-2} by considering, instead of $\gamma_{1}(\alpha,u)$, the first
eigenvalue of the problem
\[%
\begin{cases}
-\Delta\phi-(\alpha+qa(x)u^{q-1})\phi=\gamma(\alpha,u)\phi &
\mbox{in $\Omega$},\\
\partial_{\nu}\phi=0 & \mbox{on $\partial \Omega$}.
\end{cases}
\]
Let us set
\[
\tilde{\alpha}_{s}=\tilde{\alpha}_{s}(a,q):=\sup\{\alpha\geq
0:\mbox{$(S_\alpha)$ has a solution in $\mathcal{P}^\circ$}\}.
\]
Arguing as in \cite[Proposition 2.2]{alama}, one may find some $C>0$ such that
$\tilde{\alpha}_{s}(a,q)<C$ for any $q\in(0,1)$. This result contrasts with
the corresponding one for $(P_{\alpha})$, where the condition `$0\not \equiv
a\geq0$ in some smooth domain $D\subset\Omega$ such that $\left\vert \partial
D\cap\partial\Omega\right\vert >0$' is assumed, cf. \cite[Proposition
3.6]{KRQU2019}. Thus this assumption is
\textit{not} needed in item (iv) of the following result:

\begin{theorem}
\label{mthms}Assume $(A.0),$ $(A.1)$, and $q\in\mathcal{I}_{\mathcal{N}}$. Then:

\begin{enumerate}
\item $(S_{\alpha})$
has a solution curve $\mathcal{C}_{1}=\left\{  (\alpha,u_{1,\alpha}%
);0\leq\alpha\leq\tilde{\alpha}_{s}\right\}  $ such that $\alpha\mapsto
u_{1,\alpha}\in\mathcal{P}^{\circ}$ is continuous and increasing on
$[0,\tilde{\alpha}_{s}]$ and $C^{\infty}$ in $[0,\tilde{\alpha}_{s})$, with
$u_{1,0}=u_{\mathcal{N}}$, and $u_{1,\tilde{\alpha}_{s}}$ is the unique
solution of $(S_{\tilde{\alpha}_{s}})$ in $\mathcal{P}^{\circ}$. Moreover,
$\mathcal{C}_{1}$ is extended to a $C^{\infty}$ curve, say $\mathcal{C}%
_{1}^{\prime}$ ($\supset\mathcal{C}_{1}$), bending \textrm{to the left} in a
neighborhood of $(\tilde{\alpha}_{s},u_{1,\tilde{\alpha} _{s}})$.

\item $(S_{\alpha})$ has a solution curve $\mathcal{C}_{2}=\left\{
(\alpha,u_{2,\alpha});0<\alpha\leq\overline{\alpha}\right\}  $ for some
$\overline{\alpha}\in(0,\tilde{\alpha}_{s}]$, such that $\alpha\mapsto
u_{2,\alpha}\in\mathcal{P}^{\circ}$ is continuous and decreasing on
$(0,\overline{\alpha}]$ and $C^{\infty}$ in $(0,\overline{\alpha})$, with
$\min_{\overline{\Omega}}u_{2,\alpha}\rightarrow\infty$ as $\alpha
\rightarrow0^{+}$. Moreover:

\begin{enumerate}
\item for any interior point $(\alpha,u)\in\mathcal{C}_{1}^{\prime}%
\cup\mathcal{C}_{2}$, the solutions set of $(S_{\alpha})$ in a neighborhood of
$(\alpha,u)$ is given exactly by $\mathcal{C}_{1}^{\prime}\cup\mathcal{C}_{2}$ ;

\item for every $\alpha\in(0,\overline{\alpha})$ the solutions $u_{1,\alpha
},u_{2,\alpha}$ are strictly ordered
by $u_{2,\alpha}-u_{1,\alpha}\in\mathcal{P}^{\circ}$, and these ones are the
only solutions of $(S_{\alpha})$ in $\mathcal{P}^{\circ}$ for $\alpha>0$ small.
\end{enumerate}

\item $\mathcal{C}_{1}^{\prime}$ is connected to $\mathcal{C}_{2}$ by a
component (i.e., a maximal closed, connected subset) $\mathcal{C}_{\ast}$ of
solutions of $(S_{\alpha})$ in
$[0,\tilde{\alpha}_{s}]\times\mathcal{P}^{\circ}$, see Figure
\ref{fig19_0422compo}(i).
\end{enumerate}
\end{theorem}

We turn now to the analogue of Theorem \ref{ball}. We denote by $\tilde
{\alpha}_{2}=\tilde{\alpha} _{2}(\Omega)$ the second (and first nontrivial)
eigenvalue of the Neumann problem%
\[%
\begin{cases}
-\Delta\phi=\alpha\phi & \mbox{ in }\Omega,\\
\partial_{\nu}\phi=0 & \mbox{ on }\partial\Omega.
\end{cases}
\]

\begin{theorem}
Assume $(A.0),$ $(A.1)$, $q\in\mathcal{I}_{\mathcal{N}}$ and
\begin{equation}
\frac{-\int_{\Omega}a}{\int_{\Omega}u_{\mathcal{N}}^{1-q}}\leq\tilde{\alpha
}_{2}. \label{hip-s}%
\end{equation}
Then the solutions set in $\mathcal{P}^{\circ}$ of $(S_{\alpha})$ with
$\alpha\geq0$ consists of $\mathcal{C}_{1}^{\prime}$, where the upper curve of
$\mathcal{C}_{1}^{\prime}$ is given exactly by $\mathcal{C}_{2}$ with
$\overline{\alpha}=\tilde{\alpha}_{s}$. In particular, $\left(  S_{\alpha
}\right)  $ has exactly two solutions in $\mathcal{P}^{\circ}$ for all
$\alpha\in(0,\tilde{\alpha}_{s})$.
\end{theorem}

Note that the quotient in the left-hand side of \eqref{hip-s} arises exactly
as the one in \eqref{hip} in the proof of \cite[Proposition 3.4]{KRQU2019}.
Furthermore, the quotient in \eqref{hip-s} is easier to handle (i.e. to make
it small) since the integral of $u_{\mathcal{N}}^{1-q}$ is computed over
$\Omega$, whereas in \eqref{hip} it is taken over $\partial\Omega$ in
\eqref{hip}. Examples of $a$ satisfying \eqref{hip-s} can be found proceeding
as in subsection 2.1.

\appendix

\section{Reduction to a fixed point equation in $C(\overline{\Omega})$}

\label{appen}

Let us formulate $(P_{\alpha})$ as a fixed point equation to which the sub and
supersolutions method \cite[Proposition 7.8]{Am76} applies. Note that
$C(\overline{\Omega})$ is equipped with the positive cone $\mathcal{P}%
:=\left\{  u\in C(\overline{\Omega}):u\geq0\mbox{ on }\overline{\Omega
}\right\}  $, whose interior is given by $\mathcal{P}^{\circ}$.

First we introduce the solution operators of some linear boundary value
problems associated with $(P_{\alpha})$. Let $c>0$ be a constant and
$\mathcal{K}_{\Omega}:C^{\theta}(\overline{\Omega})\rightarrow C^{2+\theta
}(\overline{\Omega})$ be the \textit{solution operator} of the problem
\[%
\begin{cases}
(-\Delta+c)u=g & \mbox{ in }\Omega,\\
\partial_{\nu}u=0 & \mbox{ on }\partial\Omega,
\end{cases}
\]
i.e.,
$\mathcal{K}_{\Omega}$ is a homeomorphism, and given $g\in C^{\theta
}(\overline{\Omega})$, $\mathcal{K}_{\Omega}g$ is the unique solution of the
problem above. It is well known \cite{Am76} that $\mathcal{K}_{\Omega}$ is
uniquely extendable to a compact linear mapping from $C(\overline{\Omega})$
into $C^{1}(\overline{\Omega})$, and moreover, it is \textit{strongly
positive}, i.e., $\mathcal{K}_{\Omega}g\in\mathcal{P}^{\circ}$ for any
$g\in\mathcal{P}\setminus\{0\}$. In a similar way, we denote by $\mathcal{K}%
_{\partial\Omega}:C^{1+\theta}(\partial\Omega)\rightarrow C^{2+\theta
}(\overline{\Omega})$ the \textit{solution operator} of the problem
\[%
\begin{cases}
(-\Delta+c)u=0 & \mbox{ in }\Omega,\\
\partial_{\nu}u=h & \mbox{ on }\partial\Omega,
\end{cases}
\]
i.e., $\mathcal{K}_{\partial\Omega}$ is a homeomorphism, and given $h\in
C^{1+\theta}(\partial\Omega)$, $\mathcal{K}_{\partial\Omega}h$ is the unique
solution of this problem. It is well known \cite{Am76N} that $\mathcal{K}%
_{\partial\Omega}$ is uniquely extendable to a bounded linear mapping from
$C(\partial\Omega)$ into $C^{\theta}(\overline{\Omega})$, which is
nonnegative. Using the usual trace $\tau:C(\overline{\Omega})\rightarrow
C(\partial\Omega)$, it follows that $\mathcal{K}_{\partial\Omega}\circ
\tau:C(\overline{\Omega})\rightarrow C(\overline{\Omega})$ is compact and
nonnegative. In the sequel, $\mathcal{K}_{\partial\Omega}\circ\tau$ is still
denoted by $\mathcal{K}_{\partial\Omega}$.

Summing up, given $a\in C^{\theta}(\overline{\Omega})$ with $\theta\in(0,1)$,
$u\in C^{2+\theta}(\overline{\Omega})\cap\mathcal{P}^{\circ}$ is a solution of
$(P_{\alpha})$ if and only if $u\in\mathcal{P}^{\circ}$ solves the fixed point
equation
\[
u=\mathcal{F}_{c}(u):=\mathcal{K}_{\Omega}(cu+a(x)u^{q})+\mathcal{K}%
_{\partial\Omega}(\alpha u)\quad\mbox{ in }C(\overline{\Omega}).
\]

Next we explain how to apply \cite[Proposition 7.8]{Am76} to our setting. Let
$v,w\in\mathcal{P}^{\circ}$ be such that $w-v\in\mathcal{P}^{\circ}$, and
assume that $v\leq\mathcal{F}_{c}(v)$ and $w\geq\mathcal{F}_{c}(w)$. In view
of these inequalities and the above formulation, $v$ and $w$ are still called
a \textit{subsolution} and a \textit{supersolution} of $(P_{\alpha})$,
respectively. In particular, let us choose $c>0$ large such that
\[
m_{c}\left(  \xi\right)  :=c+qa(x)\xi^{q-1}>0\ \mbox{ on }\ \overline{\Omega
}\quad\mbox{ for }\ \xi\in\lbrack\min_{\overline{\Omega}}v,\max_{\overline
{\Omega}}w]\ \mbox{and }\ x\in\overline{\Omega}.
\]
Set the order interval
\[
\lbrack v,w]:=\left\{  u\in C(\overline{\Omega}):v\leq u\leq
w\mbox{ on }\overline{\Omega}\right\}  ,
\]
and then, under the condition above $\mathcal{F}_{c}$ is \textit{strongly
increasing} in $[v,w]$, i.e., if $u_{1},u_{2}\in\lbrack v,w]$ and $u_{2}%
-u_{1}\in\mathcal{P}\setminus\{0\}$, then $\mathcal{F}_{c}(u_{2}%
)-\mathcal{F}_{c}(u_{1})\in\mathcal{P}^{\circ}$. We apply \cite[Proposition
7.8]{Am76} to the mapping $\mathcal{F}_{c}:[v,w]\rightarrow C(\overline
{\Omega})$ to deduce that the equation above has at least one solution $u$
such that $u\in\lbrack v,w]$. So $u\in C^{2+\theta}(\overline{\Omega}%
)\cap\mathcal{P}^{\circ}$ is a solution of $(P_{\alpha})$.

For such solution $u$, we note that $m_{c}(u)\in\mathcal{P}^{\circ}$ in view
of the inequality above, so the eigenvalue problem
\[
\mathcal{F}_{c}^{\prime}(u)\psi=\mathcal{K}_{\Omega}(m_{c}\left(  u\right)
\psi)+\mathcal{K}_{\partial\Omega}(\alpha\psi)=\sigma\psi
\]
has a largest eigenvalue $\sigma_{1}>0$, which is simple and is the only
principal eigenvalue (with a corresponding eigenfunction $\psi_{1}\in
C^{2+\theta}(\overline{\Omega})\cap\mathcal{P}^{\circ}$), cf. \cite[Section 3,
Chapter 1]{Am76}. Consequently, we have
\[%
\begin{cases}
\sigma_{1}(-\Delta+c)\psi_{1}=m_{c}\left(  u\right)  \psi_{1} &
\mbox{ in }\Omega,\\
\sigma_{1}\partial_{\nu}\psi_{1}=\alpha\psi_{1} & \mbox{ on }\partial\Omega.
\end{cases}
\]

To apply the stability results in \cite[Proposition 7.8]{Am76}, we finally
compare the relation between $\sigma_{1}$ and $\gamma_{1}$, where we recall
$\gamma_{1}$ is the smallest eigenvalue of \eqref{ep}. More precisely, we
shall verify that
\begin{equation}%
\begin{cases}
& \sigma_{1}>1\ \Longleftrightarrow\ \gamma_{1}<0,\\
& \sigma_{1}<1\ \Longleftrightarrow\ \gamma_{1}>0.
\end{cases}
\tag{A}\label{a}%
\end{equation}
Recalling \eqref{ep} and using Green's formula, we obtain
\begin{align*}
\int_{\Omega}(-\Delta+c)\phi_{1}\cdot\psi_{1}-\phi_{1}\cdot(-\Delta
+c)\psi_{1}  &  =\int_{\partial\Omega}-\partial_{\nu}\phi_{1}\cdot\psi
_{1}+\phi_{1}\cdot\partial_{\nu}\psi_{1}\\
&  =\alpha\left(  \frac{1}{\sigma_{1}}-1\right)  \int_{\partial\Omega}\phi
_{1}\psi_{1}.
\end{align*}
A direct computation shows
\[
\int_{\Omega}(-\Delta+c)\phi_{1}\cdot\psi_{1}-\phi_{1}\cdot(-\Delta+c)\psi
_{1}=\gamma_{1}\int_{\Omega}\phi_{1}\psi_{1}+\left(  1-\frac{1}{\sigma
}\right)  \int_{\Omega}m_{c}\left(  u\right)  \phi_{1}\psi_{1}.
\]
Combining them provides
\[
\gamma_{1}\int_{\Omega}\phi_{1}\psi_{1}=\left(  \frac{1}{\sigma_{1}}-1\right)
\left(  \int_{\Omega}m_{c}\left(  u\right)  \phi_{1}\psi_{1}+\alpha
\int_{\partial\Omega}\phi_{1}\psi_{1}\right)  ,
\]
showing \eqref{a}.

{\small
}

{\small
} {\small
}

{\small
}

{\small
}

{\small
}

{\small
}

{\small
}

{\small
}

{\small
}

{\small
}

{\small
}

{\small
}

{\small
}

{\small
}

{\small
}

{\small
}

{\small
}

\end{document}